\documentclass[11pt]{amsart}
\usepackage{fullpage}
\usepackage{amsfonts,amssymb,amsmath,amsthm}
\usepackage{hyperref}
\usepackage{color}

\theoremstyle{plain}
\newtheorem{theorem}{Theorem}[section]
 \newtheorem{corollary}[theorem]{Corollary}
 \newtheorem{lemma}[theorem]{Lemma}
 \newtheorem{proposition}[theorem]{Proposition}
 \theoremstyle{definition}
 \newtheorem{definition}[theorem]{Definition}
 
 \theoremstyle{remark}
 \newtheorem{remark}[theorem]{Remark}
 \numberwithin{equation}{section}

\def\eps{\varepsilon}
\def\R{{\mathbb R}}
\def\C{{\mathbb C}}
\def\N{{\mathbb N}}

\def\Tend#1#2{\mathop{\longrightarrow}\limits_{#1\rightarrow#2}}

\numberwithin{equation}{section}

\begin{document}
\title[]{Quantum evolution and  sub-Laplacian  operators\\ 
on groups of Heisenberg type}
\author[C. Fermanian]{Clotilde~Fermanian-Kammerer}
\address[C. Fermanian Kammerer]{
Universit\'e Paris Est Cr\'eteil, LAMA, 61, avenue du G\'en\'eral de Gaulle\\
94010 Cr\'eteil Cedex\\ France}
\email{clotilde.fermanian@u-pec.fr}
\author[V. Fischer]{V\'eronique Fischer}\address[V. Fischer]%
{University of Bath, Department of Mathematical Sciences, Bath, BA2 7AY, UK} 
\email{v.c.m.fischer@bath.ac.uk}

\begin{abstract} In this paper we analyze the evolution of the time averaged energy densities associated with a family of solutions to a Schr\"odinger equation on a Lie group of Heisenberg type. We use a semi-classical approach adapted to the stratified structure of the group and describe the semi-classical measures (also called quantum limits) that are associated with this family. This  allows us to prove an Egorov's type Theorem describing the quantum evolution of a pseudodifferential semi-classical operator through the semi-group generated by a sub-Laplacian. 
 \end{abstract}

\subjclass[2010]{58J47, 43A80, 35Q40}
\keywords{
Analysis on nilpotent Lie groups, 
Evolution of solutions to the Schrodinger equation, 
Semi-classical analysis for sub-elliptic operators,
Abstract harmonic analysis and C*-algebra theory. 
}

\maketitle

\tableofcontents

\section{Introduction}\label{intro}

 We consider groups of Heisenberg type, or $H$-type groups $G$, 
 which are a special case of  simply connected Lie groups  stratified of step~$2$ as described more precisely later. 
 As a step~$2$ stratified group, 
  its  Lie algebra~${\mathfrak g}$  is  equipped with a vector space decomposition
$$  
  \displaystyle
  {\mathfrak g}=  \mathfrak v \oplus \mathfrak z \, ,
  $$
  such that $[{\mathfrak v},{\mathfrak v}]= {\mathfrak z}\not=\{0\}$ and ${\mathfrak z}$ is the center of ${\mathfrak g}$.
  Choosing a basis $V_j$ of ${\mathfrak v}$ and identifying $\mathfrak g$ with the Lie algebra of left-invariant vector fields on $G$, 
  one defines the sublaplacian

    $$
  \Delta_G= \sum_{1\leq j\leq {\rm dim}\, {\mathfrak v}} V_j^2
  $$
  together with the associated Schr\"odinger propagator ${\rm e}^{it\Delta_G}$.  
We are interested in the asymptotic analysis as~$\eps$ goes to $0$
of quantities of the form 
\begin{equation}\label{def:timeavdens}
{1\over T} \int_0^{T} \int_G \phi(x)| {\rm e}^{i {t\over  2 \eps^{\aleph} } \Delta_G}\psi^\eps_0(x)|^2 dx\,dt
\end{equation}
for $\phi\in{\mathcal C}_c^\infty(G)$, $T\in\R$, $\aleph\in \R$
 and $(\psi^\eps_0)_{\eps>0}$ a bounded family of $L^\infty(\R,L^2(G))$
which satisfies
\begin{equation}
\label{Hscriterium}
\exists s,C_s>0,\qquad \forall \eps>0\qquad   \eps^{s} \| (-\Delta_G)^{s\over 2} \psi^\eps_0\|_{L^2(G)}+  \eps^{-s} \| (-\Delta_G)^{-{s\over 2} }\psi^\eps_0\|_{L^2(G)}\leq C_s,
\end{equation}
so that the oscillations of the initial data are exactly of size $1/\eps$. Taking into account that the operator $\Delta_G$ is homogeneous of degree~$2$ and writing 
$$
{t\over \eps^\aleph} \Delta_G= {t\over\eps^{\aleph +2} } \eps^2 \Delta_G,
$$
 we choose $\aleph >-2$.
 Considering the asymptotics $\eps\rightarrow 0$ then consists in doing an analysis in large times (times of  sizes $O\left({ \eps^{-\aleph -2}}\right)$) simultaneously  with the study of the dispersion of the concentration or oscillation effects that are present in the initial data. 
 A consequence of our main  results is the next theorem where we denote by ${\mathcal M}^+(Z)$ the set of finite positive Radon measures on a locally compact Hausdorff set~$Z$ (see Definition \ref{def_LinftyRX+} for  
  $L^\infty(\R,\mathcal M^+(Z))$).

\begin{theorem}
\label{theorem1}
Let $G$ be a H-type group and $(\psi^\eps_0)_{\eps>0}$ a bounded family of $L^\infty(\R,L^2(G))$ satisfying~\eqref{Hscriterium}.
Any weak limit of the measure 
$| {\rm e}^{i  {t\over 2\eps^{\aleph} } \Delta_G}\psi^\eps_0(x)|^2 dx\,dt$
 is of the form $d\varrho_t(x)\otimes dt$
where $t\mapsto \varrho_t$ is a  map in $L^\infty(\R,\mathcal M^+(G))$.
Moreover, for almost all $t\in \R$,
the measure $\varrho_t$ writes 
$$\varrho_t= \varrho_t^{\mathfrak v^*}+ \varrho_t^{\mathfrak z^*}, \qquad\mbox{with}\quad
t\longmapsto\varrho_t^{\mathfrak z^*} \ \mbox{and} \  t\longmapsto \varrho_t^{\mathfrak v^*} \ \mbox{in}\ L^\infty(\R,{\mathcal M}^+(G)),$$ 
and has the following properties. 
\begin{itemize}
\item  If $\aleph \in(-2,-1)$, 
 for all $t\in \R$, 
 $\varrho_t^{\mathfrak v^*}=\varrho_0^{\mathfrak v^*}$ and $\varrho_t^{\mathfrak z^*}=\varrho_0^{\mathfrak z^*}$ are independent of the time $t\in \R$, and $\varrho_0^{\mathfrak v^*}+\varrho_0^{\mathfrak z^*}$ is equal to a weak limit of $|\psi^\eps_0(x)|^2 dx$.
\item If $\aleph=-1$, 
then $\ \varrho^{\mathfrak z^*}_t = \varrho^{\mathfrak z^*}_0$
where $\varrho^{\mathfrak z^*}_0 \in \mathcal M^+(G)$  depends only on  $(\psi^\eps_0)$, 
and 
\begin{equation}\label{eq:rhoEuclidean}
\varrho_t^{\mathfrak v^*} (x) =
\int_{\mathfrak v^*} 
\varsigma_0 \left({\rm Exp} (t\,  \omega\cdot  V ) x,d\omega\right).
\end{equation}
for  $\varsigma_0\in{\mathcal M}^+(G\times \mathfrak v^*)$ which depends only on  $(\psi^\eps_0)$.
\item  If $\aleph \in (-1,0)$, then $\varrho^{\mathfrak v^*}_t=0$ and 
 $\partial_t \varrho^{\mathfrak z^*}_t=0$ holds in the sense of distributions on $\R \times G$.
 \item If $\aleph=0$, then $\varrho^{\mathfrak v^*}_t=0$ and 
$\varrho^{\mathfrak z^*}_t$ decomposes into
$$
\varrho_t^{\mathfrak z^*} = \sum_{n\in \N} 
\int_{{\mathfrak z}^*\setminus\{0\}} 
\gamma_{n,t}(x,d\lambda) 
$$
where $t\mapsto \gamma_{n,t}$ is in $L^\infty(\R,{\mathcal M}^+(G \times (\mathfrak z^*\setminus\{0\}))$;
furthermore we have  in the sense of distributions on $\R \times G \times (\mathfrak z^*\setminus\{0\})$,
$$
\left(\partial_t -{2n+d\over 2|\lambda|} {\mathcal Z}^{(\lambda)}\right)\gamma_{n,t}=0,
$$
where  ${\mathcal Z}(\lambda)$ is the left invariant vector field corresponding to $\lambda\in{\mathfrak z}^*$.
\item	If $\aleph>0$, then  $\varrho_t=0$ for all $t\in \R$.
\end{itemize}
\end{theorem}

Several aspects are interesting to notice.
Firstly, there exists a threshold, $\aleph=0$, above which the  weak limits of the time-averaged  energy density is~$0$; 
this means that for sufficiently  large scale of times, all the concentrations and oscillations effects have disappeared: the dispersion is complete. A similar picture holds in the Euclidean setting, however the threshold occurs at $\aleph=-1$ (see~\cite{AFM} and the Appendix in this article). This  illustrates the fact that the dispersion is slower in sub-Riemanian geometries  than in Euclidean ones, as already noticed in~\cite{BGX,hiero,BFG2}. 
Secondly, one observes a decomposition  of these weak limits   into two parts $\varrho_t=\varrho^{\mathfrak v^*}_t+\varrho^{\mathfrak z^*}_t$ which turn out to have different transitional regimes: 
$\aleph=-1$ for $\varrho_t^{\mathfrak v^*}$ and $\aleph=0$ for $\varrho_t^{\mathfrak z^*}$. This splitting is also present in the works~\cite{BS} about Grushin-Schr\"odinger equation and~\cite{Zeld97,CdVHT} about sublaplacians on contact manifolds. The part $\varrho_t^{\mathfrak v^*}$ behaves like in the Euclidean setting and equation~(\ref{eq:rhoEuclidean}) also presents Euclidean features. However, the other part~$\varrho_t^{\mathfrak z^*}$ looks completely different and is specific to the nilpotent Lie group context,  showing that the structure of the limiting objects is more complex than in the Euclidean case.  

\medskip

Similar questions have been addressed for the Laplace operator in different geometries, including compact ones: in the torus and for integrable systems (\cite{AM:12,AFM}), in Zoll manifolds (see \cite{MR:16,MaciaReview} and the review~\cite{MaciaAv}), or on manifolds such as the sphere (\cite{MaciaDispersion}). 
In contrast with the non-compact  case (which is ours here), 
the compactness of the manifold implies that  the complete dispersion of the energy is not possible;
furthermore, the weak limits  of the energy densities possess structural properties due to the geometry of the manifold, such as invariance by some flows, that may allow for their determination. 
For example, on compact Riemanian manifolds, 
such a measure belongs to the set of measures which are invariant under the geodesic flow, and this property is at the root  of quantum ergodicity theorem~\cite{Schni,CdV,zeldi} (see the introductory survey~\cite{AFF} and the articles~\cite{ZZ,DJN} for more recent developments in the topic). 
The question of quantum ergodicity also arises in sub-riemanian geometries and have been addressed for contact~\cite{Zeld97,CdVHT} and quasi-contact~\cite{Savale} manifolds. As will be made precise in the next sections, we observe invariance properties by a flow that turns out to coincide with the Reeb flow used in~\cite{Zeld97,CdVHT}  when $G$ is the Heisenberg group.

\medskip 

Theorem~\ref{theorem1} is a consequence of the  main results of this paper which use the semi-classical approach introduced  in~\cite{FF2} for H-type Lie groups  and are in the spirit of the article~\cite{MaciaReview} for the treatment of the large time evolution together with the oscillations. 
They are as follows:
\begin{enumerate}
\item
The first result is an Egorov's type Theorem on H-type groups  (see Theorem~\ref{theo:Egorov}), which describes as~$\eps$ goes to~$0$ the asymptotics of  quantities of the form
\begin{equation}
\label{eq:egorov}
\int_{\R} \theta(t) \left( {\rm e}^{-i {t\over 2\eps^{\aleph} } \Delta_G} {\rm Op}_\eps (\sigma) {\rm e}^{i {t\over 2 \eps^{\aleph} } \Delta_G} f,f\right)_{L^2(G)} dt
 \end{equation}
 for $\theta\in{\mathcal C}_c^\infty(\R)$, $f\in L^2(G)$ and where the operator $ {\rm Op}_\eps (\sigma) $ is the semi-classical operator of a symbol~$\sigma$ as introduced in~\cite{FF2} (see also \cite{FF, FR, BFG1, Taylor84})
 All these elements are carefully explained in Section~\ref{sec:result1Egorov}.\\
\item The second result concerns the structure of the limiting objects when passing to the limit in~(\ref{eq:egorov}). We extend the notion of semi-classical measure introduced in~\cite{FF2} to a time-dependent context and analyze the  properties of the semi-classical measures associated in that manner with the family $
({\rm e}^{i {t\over 2 \eps^{\aleph}  } \Delta_G}\psi^\eps_0)_\eps$, depending on the value of~$\aleph$. We give a complete description of these limiting objects in Theorems~\ref{theo:mesures} and~\ref{theo:schro} below. \\
\end{enumerate}

The proof of Theorem~\ref{theorem1} is
based on the fact that, under certain hypothesis on the size of the oscillation, 
the analysis of the weak limits of the energy density can be deduced from those of its semi-classical measures, which are also called quantum limits in some geometric contexts. This  idea was
introduced in the 90's  in the Euclidean case (see~\cite{HelfferMartinezRobert,gerard_X,LionsPaul}),  and adapted for $H$-type groups in~\cite{FF2}.
The hypothesis on the size of the oscillation of $(\psi_0^\eps)_\eps$ is  a {\it uniform strict $\eps$-oscillation} property (see Section~\ref{subsec:epsosc})
which guarantees that the oscillations are of sizes $\eps^{-1}$
and is implied by the condition in~(\ref{Hscriterium}).   Then,
using the semi-classical pseudodifferential operators constructed in~\cite{FF2}, we determine the semi-classical measures that are associated with the family ${\rm e}^{i {t\over2 \eps^{\aleph} } \Delta_G}\psi^\eps_0$ and prove Theorem~\ref{theorem1}. 

\medskip

 A straightforward generalization of our result would consist in adding a scalar  
 potential $\eps^\theta V(x)$ for a smooth function $V$ defined on $G$ and a parameter $\theta\in \R^+$. Then, one could exhibit regimes depending on the position of $\theta$ with respect to $\aleph$ and the vector fields to consider should be modified in a non-trivial manner. One should then consider operations on symbols $\sigma(x,\lambda)$  that involve differentials of the potential~$V(x)$ and difference operators acting on the operator part of $\sigma(x,\lambda)$.  
 A second generalization would be to consider more general stratified and graded groups. 
 This would require to obtain in this more general setting similar results to those of Appendix B which at the moment heavily rely on the special case of $H$-type groups. 
 However, the authors think this is doable and they have this generalization in mind. They also think that this approach can be adapted to homogeneous spaces.

\medskip 

In the next section, we recall the definition of $H$-type groups and present our two  main results shortly described  above,  the Egorov theorem~\ref{theo:Egorov}
and the analysis of the semi-classical measures associated with a family of the Schr\"odinger equation
in Theorem~\ref{theo:schro}. 
After some preliminary results on  semi-classical symbols in Section~\ref{sec:3.1},
we prove both theorems in Section~\ref{sec:proofs}.
Theorem~\ref{theorem1} is a consequence of this analysis and is proved in Section~\ref{sec:densitylimit}. An Appendix is devoted to a short description of the Euclidean case and to some technical auxiliary results.


 \section{Main results}\label{sec:result1Egorov}

  \subsection{H-type groups, notations and definitions}
 \label{subsec_preliminaries}
A simply connected Lie group  $G$  is said to be stratified of step~$2$
  if its Lie algebra~${\mathfrak g}$ is  equipped with a vector space decomposition
$$  
  \displaystyle
  {\mathfrak g}=  \mathfrak v \oplus \mathfrak z \, ,
  $$
  such that $[{\mathfrak v},{\mathfrak v}]= {\mathfrak z}\not=\{0\}$ and ${\mathfrak z}$ is the center of ${\mathfrak g}$. 
  Via the exponential map  
 $$
 {\rm exp} :  {\mathfrak g} \rightarrow G 
 $$ 
 which is a diffeomorphism from ${\mathfrak g}$ to $G$, 
 one identifies $G$ and ${\mathfrak g}$ as a set and a manifold. 
 Under this identification, 
 the  group law on $G$ (which is generally not commutative) is provided by the Campbell-Baker-Hausdorff formula, 
 and $(x,y)  \mapsto x  y  $ is a polynomial map.  
 More precisely, if $x={\rm Exp}(v_x+z_x)$ and 
 $y={\rm Exp}(v_y+z_y)$ then
 $$
 xy ={\rm Exp} (v+z),\;\; v=v_x+v_y\in{\mathfrak v},\;\;z= z_x+z_y+\frac 12 [v_x,v_y]\in{\mathfrak z}.
 $$
If $x={\rm Exp}(v)$ then $x^{-1}={\rm Exp}(-v)$. We may identify ${\mathfrak g}$ with the space of left-invariant vector fields via 
 $$
 Xf(x)= \left.{d\over dt} f({\rm Exp}(tX)x)\right|_{t=0},
 \quad x\in G.
 $$
    \medskip
    
For any~$\lambda \in  \mathfrak z^\star$ (the dual of the center~$  \mathfrak z$) we define a skew-symmetric bilinear form on $\mathfrak v$ by 
\begin{equation}
\label{skw}
\forall \, U,V \in  \mathfrak v \, , \quad B(\lambda) (U,V):= \lambda([U,V]) \, .
 \end{equation}
 Following~\cite{Kaplan80}, we say that $G$ is of H-type (or of Heisenberg type) when, once the inner products on~${\mathfrak v}$ and  on~${\mathfrak z}$ are fixed, the endomorphism of this skew symmetric form (that we still denote by $B(\lambda)$) satisfies 
 \begin{equation}\label{eq:B(lambda)}
 \forall \lambda\in{\mathfrak z}^*,\;\; B(\lambda)^2=-|\lambda|^2 {\rm Id}_{\mathfrak v}.
 \end{equation}
 This implies in particular that the dimension of ${\mathfrak v}$ is even. We set 
 $$
 {\rm dim}\, {\mathfrak v} =2d,\;\; {\rm dim} \,{\mathfrak z}=p.
 $$
 
 \subsubsection{Orthonormal basis of ${\mathfrak g}$}
 \label{subsubsec_ONBg}
 One can find an orthonormal basis
  $\left (P_1 , \dots ,P_d,  Q_1 , \dots ,Q_d\right)$ 
  where $B(\lambda)$ is represented by 
    \begin{equation}
\label{eqBJ}  
  B(\lambda)(U,V)= |\lambda| U^t JV, \quad\mbox{where}\quad 
J=\begin{pmatrix}0 & {\rm Id} \\ -{\rm Id} & 0\end{pmatrix}.
  \end{equation}
for two vectors $U,V\in \mathfrak v$ written in the $\left (P_1 , \dots ,P_d,  Q_1 , \dots ,Q_d\right)$-basis.
We decompose~$ \mathfrak v$ in a $\lambda$-depending way as
$ \mathfrak v = \mathfrak p_\lambda+  \mathfrak q_\lambda$
 with 
 $$
 \begin{aligned}
  \mathfrak p:=\mathfrak p_\lambda:= \mbox{Span} \, \big (P_1, \dots ,P_d \big) \, , & \quad \mathfrak q:=\mathfrak q_\lambda:= \mbox{Span} \, \big (Q_1, \dots ,Q_d\big).
   \end{aligned}
$$
The fundamental property~(\ref{eq:B(lambda)}) satisfied by  $B(\lambda)$ considered as an endomorphism on ${\mathfrak v}$ implies that for all $V\in{\mathfrak v}$, 
$| B(\lambda)V|_{\mathfrak v}^2 = |\lambda|^2 | V|_{\mathfrak v}^2$, and, by linearization, we deduce
$$
\forall \,U,U'\in {\mathfrak v},\;\; 
\forall\, \lambda,\lambda'\in{\mathfrak z}^*, \;\; (B(\lambda)U,B(\lambda')U')_{\mathfrak v}
=
 (\lambda,\lambda')_{{\mathfrak z}^*}(U,U')_{\mathfrak v}.
$$
As $(B(\lambda)U,B(\lambda')U')_{\mathfrak v} = (\lambda,[U,B(\lambda')U']\rangle_{{\mathfrak z}^*,{\mathfrak z}}$, 
we deduce for any $\lambda\in {\mathfrak z}^*\setminus\{0\}$ 
\begin{equation} 
\label{eq_Zlambda}
\forall j=1,\ldots,d,\qquad
[P_j, Q_j]=|\lambda|^{-1}{\mathcal Z}^{(\lambda)}, 
\end{equation}
where ${\mathcal Z}^{(\lambda)}$ is the unique vector of ${\mathfrak z} $ equal to $\lambda$  through the identification of ${\mathfrak z}^*$ to ${\mathfrak z}$ via the inner product,
and for all $1\leq j_1,j_2 \leq d$
\begin{equation} 
\label{eq_comPQ}	
  j_1\not= j_2 \ \Longrightarrow \  
[P_{j_1}, P_{j_2}]=0,\quad
[Q_{j_1}, Q_{j_2}]=0,\quad
[P_{j_1}, Q_{j_2}]=0.
\end{equation}

\subsubsection{Realisation of the elements in $G$}
Denoting by $z=(z_1,\cdots ,z_p)$ the coordinate of $Z$ in a fixed orthonormal basis $(Z_1,\cdots, Z_p)$ of $\mathfrak z$, and once given $\lambda\in\mathfrak z^*$,
we will often use the writing of an element $x\in G$ or $X\in \mathfrak g$ as 
\begin{equation}
\label{eqxpqz}
x={\rm Exp}(X) , \qquad X=p_1P_1 +\ldots + p_dP_d \ + \ q_1Q_1 + \ldots + q_d Q_d \ + \ z_1 Z_1 +\ldots + z_p Z_p,
\end{equation}
where $p=(p_1,\cdots,p_d)$ are the $\lambda$-dependent coordinates of $P$ on the vector basis $(P_1, \cdots,P_d)$, by $q=(q_1,\cdots,q_d)$ those of $Q$ on $(Q_1,\cdots,Q_d)$, while the coordinates $z=(z_1,\cdots , z_p)$ of $Z$ are independent of $\lambda$.
We will also fix an orthonormal basis $(V_1,\ldots, V_{2d})$ of $\mathfrak v$ to write the coordinates 
$$
v=(v_1,\ldots,v_{2d}), 
\quad\mbox{of an element}\quad
V=v_1V_1 +\ldots + v_{2d} V_{2d}
$$
of  $\mathfrak v$
independently of $\lambda$.

\subsubsection{Functional spaces on $G$}
The inner products on $\mathfrak v$ and ${\mathfrak z}$ allow us to consider the Lebesgue measure $dv\, dz$ on ${\mathfrak g}={\mathfrak v}\oplus{\mathfrak z}$. Via the identification of $G$ with ${\mathfrak g}$ by the exponential map, this induces a Haar measure $dx$ on $G$. This measure is invariant under left and right translations:
  $$
  \forall f  \in L^1(G) \, ,  \quad  \forall x  \in G \,, \quad \int_G f(y) dy  = \int_G f(x  y)dy= \int_G f(y  x)dy \, .
   $$
Note that  the convolution of two functions $f$ and $g$ on $G$ is given by
    \begin{equation}
\label{convolutiondef}
  f*g(x) :=  \int_G f(x  y^{-1})g(y)dy = \int_G f(y)g(y^{-1}   x)dy
   \end{equation}
  and  as in the Euclidean case we define   Lebesgue spaces by
$$
 \|f\|_{L^q (G)}  := \left( \int_G |f(y)|^q \: dy \right)^\frac1q \, ,
 $$
 for $q\in[1,\infty)$, with the standard modification when~$q=\infty$.
 
 \medskip 

We   define the Schwartz space~${\mathcal S}(G)$   as the set of smooth functions on~$G$ such that
 for all~$ \alpha,\beta$ in~${\mathbb N}^{2d+p}$,  
 the function
 $ x\mapsto x^\beta   {\mathcal X}^{\alpha}f(x) $ belongs to~$ L^\infty(G),
  $ where~${\mathcal X}^{\alpha}$ denotes a product of~$|\alpha|$  left invariant vector fields forming a basis of~${\mathfrak g}$ and $x^\beta$ a product of $|\beta|$ coordinate functions on $G\sim \mathfrak v \times \mathfrak z$. The Schwartz space~${\mathcal S}(G)$ can be naturally identified with  the Schwartz space~${\mathcal S}(\R^{2d+p})$;
  in particular, it is dense in Lebesgue spaces.

\subsubsection{Dilations}

  Since $G$ is stratified,   there is a natural family of dilations on ${\mathfrak g}$ defined for $t>0$ as follows: if~$X$ belongs to~$ {\mathfrak g}$, we  decompose~$X$ as~$\displaystyle X=V+Z$ with~$V\in {\mathfrak v}$ and~$Z\in {\mathfrak z}$ and we set
   $$
   \delta_t X:=tV+t^2Z  \, .
   $$
 This allows us to  define the dilation on the Lie group $G$ via the identification by the exponential map:
 $$
 \begin{array}{ccccc}
& {\mathfrak g} &\overset{\delta_t} \longrightarrow & {\mathfrak g}&\\
 {\small\rm exp}&  \downarrow& & \downarrow& {\small\rm exp}\\
  &G & \overset{{\rm exp}\, \circ\,  \delta_t \, \circ\,  {\rm exp}^{-1}}{\longrightarrow}&G
  \end{array}$$
To simplify the notation, 
 we shall still denote by $\delta_t$ the map ${\rm exp}\, \circ \delta_t \, \circ {\rm exp}^{-1}$.
 The dilations $\delta_t$, $t>0$, on $\mathfrak g$ and $G$ form a one-parameter group of automorphisms of the Lie algebra $\mathfrak g$ and of the group~$G$.
The Jacobian of the dilation $\delta_t$ is $t^Q$ where
 $$Q:={\rm dim}\, {\mathfrak v} +2{\rm dim}\, {\mathfrak z} = 2d+2p$$
  is called the homogeneous dimension of $G$.
 A differential operator $T$ on $G$
(and more generally any operator $T$ defined on $C^\infty_c(G)$ and valued in the distributions of $G\sim \R^{2d+p}$) 
 is said to be homogeneous of degree $\nu$ (or $\nu$-homogeneous) when 
 $
 T (f\circ \delta_t) = t^\nu (Tf)\circ \delta_t. 
 $

\subsection{The irreducible unitary representations and the Fourier transform}  
\subsubsection{Irreducible unitary  reresentations}
For  $\lambda\in \mathfrak z^*\setminus\{0\}$, 
the irreducible unitary representation $\pi^{\lambda}_{x} $ of~$G$ on~$L^2( \mathfrak p_\lambda)$ is defined by  
\begin{equation*}\label{def:pilambdanu}
\pi^{\lambda}_{x} \Phi(\xi)=
 {\rm exp}\left[{i\lambda(z)+ \frac i2 |\lambda|\,p q +i\sqrt{|\lambda|} \,\xi q} \right]\Phi \left(\xi+\sqrt{|\lambda|}p\right),
 \end{equation*}
 where $x$ has been written as in \eqref{eqxpqz}.
 The representations $\pi^\lambda$, $\lambda\in \mathfrak z^*\setminus\{0\}$, are infinite dimensional. 
 The other unitary irreducible representations of $G$
  are given by the characters of the first stratum in the following way:
 for every  $\omega\in \mathfrak v ^*$, 
  we set
$$
\pi^{0,\omega}_x= {\rm e}^{i \omega(V)}, \quad
x={\rm Exp} (V+Z)\in G, \quad\mbox{with}\ V\in{\mathfrak v} \ \mbox{and} \  Z\in{\mathfrak z}.
$$  
 The set $\widehat G$ of all unitary irreducible representations modulo unitary equivalence 
is then parametrized by $({\mathfrak z}^*\setminus \{0\})\sqcup {\mathfrak v}^*$: 
\begin{equation}
\label{eq_widehatG}	
\widehat G = 
\{\mbox{class of} \ \pi^\lambda \ : \ \lambda \in \mathfrak z^* \setminus\{0\}\} 
\sqcup \{\mbox{class of} \ \pi^{0,\omega} \ : \ \omega \in \mathfrak v^* \}.
\end{equation}
We will often identify each representation $\pi^\lambda$ with its equivalence class; in this case, we may write~$\mathcal H_\lambda$ for the Hilbert space of the representation instead of $L^2(\mathfrak p_\lambda) \sim L^2(\R^d)$; we also set ${\mathcal H}_{(0,\mu)}=\C$. 
Note that the trivial representation $1_{\widehat G}$ corresponds to the class of $\pi^{(0,\omega)}$ with 
$\omega=0$, i.e.
$
1_{\widehat G}
:=
\pi^{(0,0)}.
$

\subsubsection{The  Fourier transform}
\label{subsec_F}

In contrast with the Euclidean case, the Fourier transform  is defined on~$\widehat G$ and is valued  in   the space of  bounded operators
on~$L^2( \mathfrak p_\lambda)$. More precisely, the Fourier transform of a function~$f$ in~$L^1(G)$ is  defined  as follows: 
for any~$\lambda\in{\mathfrak z}^*$, $\lambda\not=0$,
 $$
 \widehat f(\lambda):=
{\mathcal F}f(\lambda):=
\int_G f(x)\left( \pi^{\lambda}_{x }\right)^* \, dx \, ,
$$
Note that    for any~$\lambda\in{\mathfrak z}^*$, $\lambda\not=0$, we have $\left( \pi^{\lambda}_{x }\right)^* =\pi^{\lambda}_{x^{-1} }$ and the map~$\pi^{\lambda}_{x}$
 is a group homomorphism from~$G$ into the group~$U (L^2( \mathfrak p_\lambda))$  of unitary operators
of~$L^2( \mathfrak p_\lambda)$, so functions~$f$ of~$L^1(G)$  have a Fourier transform~$\left({\mathcal F}(f)(\lambda)\right)_{\lambda}$ which is a bounded family of bounded operators on~$L^2( \mathfrak p_\lambda)$
 with uniform bound:
\begin{equation}
\label{eq_FfnormL1}
\|\mathcal Ff (\lambda) \|_{{\mathcal L}(L^2(\mathfrak p_\lambda))}
\leq
\int_G |f(x)|\|(\pi^\lambda_x)^* \|_{{\mathcal L}(L^2(\mathfrak p_\lambda))} dx
=
\|f\|_{L^1(G)}.
\end{equation}
since the unitarity of $\pi^\lambda$ implies $\|(\pi^\lambda_x)^* \|_{{\mathcal L}(L^2(\mathfrak p_\lambda))}=1$.

\subsubsection{Plancherel formula}
The Fourier transform can be extended to an isometry from~$L^2(G)$ onto the Hilbert
space of measurable families~$ A  = \{ A (\lambda ) \}_{(\lambda) \in{\mathfrak z}^*\setminus \{0\}}$
 of operators on~$L^2( \mathfrak p_\lambda)$ which are
Hilbert-Schmidt for almost every~$\lambda\in{\mathfrak z}^*\setminus \{0\}$,  with norm
\[ \|A\| := \left( \int_{\mathfrak z^* \setminus\{0\}}
\|A (\lambda )\|_{HS (L^2( \mathfrak p_\lambda))}^2 |\lambda|^d \, d\lambda
\right)^{\frac{1}{2}}<\infty  \, .\]
  We have the following Fourier-Plancherel formula: 
 \begin{equation}
\label{Plancherelformula} \int_G  |f(x)|^2  \, dx
=  c_0 \, \int_{\mathfrak z^* \setminus\{0\}} \|{\mathcal F}f(\lambda)\|_{HS(L^2( \mathfrak p_\lambda))}^2 |\lambda|^d  \,  d\lambda   \,, 
\end{equation}
where $c_0>0$ is a computable constant.
This yields  an inversion formula for any~$ f \in {\mathcal S}(G)$ and~$x\in G$:
\begin{equation}
\label{inversionformula} f(x)
= c_0 \, \int_{\mathfrak z^* \setminus\{0\}} {\rm{Tr}} \, \Big(\pi^{\lambda}_{x} {\mathcal F}f(\lambda)  \Big)\, |\lambda|^d\,d\lambda \,,
\end{equation}
where ${\rm Tr}$ denotes the trace of operators of ${\mathcal L}(L^2({\mathfrak p}_\lambda))$.
This   formula makes sense since for $f \in {\mathcal S}(G)$, the operators ${\mathcal F}f(\lambda)$, $\lambda\in \mathfrak z^*\setminus\{0\}$, are trace-class and $\int_{\mathfrak z^* \setminus\{0\}} {\rm{Tr}} \, \Big| {\mathcal F}f(\lambda)  \Big|\, |\lambda|^d\,d\lambda$ is finite.

\subsubsection{Fourier transform and finite dimension representations}
Usually, the Fourier transform of a locally compact group $G$ would be defined on~$\widehat G$, the set of unitary irreducible representations of~$G$ modulo equivalence, via
$$
\widehat f(\pi)=
\mathcal F f(\pi) = \int_G f(x) \pi(x)^* dx, 
$$
for a representation $\pi$ of $G$, and then considering the unitary equivalence we obtain a measurable field of operators $\mathcal F f(\pi)$, $\pi\in \widehat G$.
Here, the Plancherel measure is supported in the subset $\{\mbox{class of} \ \pi^\lambda \ : \ \lambda \in \mathfrak z^* \setminus\{0\}\}$ 
 of $\widehat G$ (see \eqref{eq_widehatG})
 since it is $c_0|\lambda|^d d\lambda$.
 This allows us to identify~$\widehat G$ and~${\mathfrak z}^*\setminus\{0\}$ when considering measurable objects up to null sets for the Plancherel measure. 
 However, our semiclassical analysis will lead us to consider objects which are also supported in the other part of $\widehat G$.
 For this reason, we also set for $\omega\in \mathfrak v^*$ and $f\in L^1(G)$:
 $$
\widehat f(0,\omega)=
 \mathcal F f (0,\omega) 
 := \int_G f(x) (\pi^{(0,\omega)}_x)^* dx
=\int_{\mathfrak v \times \mathfrak z} 
f({\rm Exp}(V+Z) ) e^{-i\omega(V)} dV dZ.	
$$

\subsubsection{Convolution and Fourier operators}
\label{subsubsec_convop}
The Fourier transform  sends the    convolution, whose definition is recalled in~(\ref{convolutiondef}),
to  composition in the following way:
\begin{equation}\label{fourconv}
 {\mathcal F}( f \star g )( \lambda ) 
 = {\mathcal F} g( \lambda)\
 {\mathcal F}f ( \lambda) \, .
 \end{equation}
We recall that a convolution operator $T$ with integrable convolution kernel $\kappa\in L^1(G)$ is defined by  $Tf=f*\kappa$ and we have $\mathcal F (Tf) = \mathcal F \kappa \, \mathcal F f$ by \eqref{fourconv}; hence, $T$ appears as  a {Fourier multiplier} with Fourier symbol $\mathcal F \kappa$ acting on the left of $\mathcal F f$.
Consequently,
 $T$ is invariant under left-translation and  bounded on $L^2(G)$ with operator norm 
$$\| T\|_{{\mathcal L} (L^2(G))}\leq \sup_{\lambda\in \widehat G} \|\mathcal F \kappa(\lambda) \|_{{\mathcal L}(L^2(\mathfrak p_\lambda))}.$$
In other words, $T$ is in the space $ {\mathcal L} (L^2(G))^G$ of the left-invariant bounded operators on $L^2(G)$.

\subsubsection{The von Neumann algebra of the group} 
Let us denote by $L^\infty(\widehat G)$
the space of bounded symbols, that is, here, measurable fields of operators $\sigma=\{\sigma(\lambda):\lambda \in \widehat G\}$ 
which are bounded
in the sense that the essential supremum 
for the Plancherel measure $c_0|\lambda|^d d\lambda$
$$
\|\sigma\|_{L^\infty (\widehat G)}
:={\rm supess}_{\lambda\in \widehat G} \|\sigma\|_{{\mathcal L}(\mathcal H_\lambda)}
$$ 
 is finite.
The space $L^\infty(\widehat G)$ is naturally equipped with a von Neummann algebra, and is called  the {\it von Neumann algebra of the group}.
As explained above, 
we already know  $L^\infty(\widehat G) \supset \mathcal F L^1(G)$ by \eqref{eq_FfnormL1},  but this inclusion is strict.

 The full Plancherel theorem \cite{Dixmier_C*} implies  that the von Neumann algebras 
$L^\infty(\widehat G)$ and the space 
 $ {\mathcal L} (L^2(G))^G$ of left-invariant bounded operators on $L^2(G)$ introduced above
  are isomorphic via
the mapping $\sigma \mapsto {\rm Op}_1(\sigma)$ 
where ${\rm Op}_1(\sigma)$ is the operator with Fourier operator symbol $\sigma$, 
$$
\mbox{i.e.}\quad
\mathcal F \left({\rm Op}_1(\sigma)f \right) = \sigma \ \mathcal F f, 
\quad f\in L^2(G).
$$
The isomorphism between $L^\infty(\widehat G)$ and 
 $ {\mathcal L} (L^2(G))^G$  allows us to naturally extend the group Fourier transform to distributions $\kappa\in \mathcal S'(G)$ such that the convolution operator $f\mapsto f* \kappa$ is bounded on $L^2(G)$ 
by setting that $\mathcal F (\kappa)$  is the symbol of the corresponding operator in $ {\mathcal L} (L^2(G))^G$.

\subsubsection{Infinitesimal representations and  Fourier transforms of left-invariant vector fields} 
The group Fourier transform can also be extended to certain classes of distributions whose convolution operators yield left-invariant operators.
Indeed, denoting by $\pi(X)$ the infinitesimal representation of $\pi$ at $X\in \mathfrak g$, i.e. $\pi(X) = \frac{d}{dt} \pi ({\rm Exp}(tX))|_{t=0}$, we have
\begin{equation*}\label{F(V)}
\mathcal F(Xf)(\pi) = \pi(X) \mathcal Ff (\pi);
\end{equation*}
here, (the class of) $\pi$ is equal to (the class of) $\pi^\lambda$ or $\pi^{(0,\omega)}$ identified with $\lambda$ or $\omega$ respectively.
For instance, we have for $j=1,\ldots,p$
\begin{equation*}
\label{def:Z}
 {\mathcal F} (Z_jf)(\lambda)=i\lambda_j {\mathcal F}f(\lambda),\quad 
 \mbox{or in other words}\quad  \pi^\lambda( Z_j )= i\lambda_j.
 \end{equation*}
We also compute for any $\lambda\in {\mathfrak z}^*\setminus\{0\}$
 \begin{equation}
\label{eq_piPQZ}	
\pi^\lambda(P_j)=\sqrt{|\lambda|}\partial_{\xi_j},
\qquad
\pi^\lambda(Q_j)=i\sqrt{|\lambda|}{\xi_j}
\qquad\mbox{and}\qquad
\pi^\lambda({\mathcal Z}^{(\lambda)})=i|\lambda|^2,\end{equation}
and for $\omega\in{\mathfrak  v}^\star$ and $j\in\{1,\cdots, p\}$,
\begin{equation}\label{Fourieromega}
\pi^{0,\omega} (V_j)=i\omega_j\;\;\mbox{and}\;\; \pi^{0,\omega}({\mathcal Z}^{(\lambda)})=0.
\end{equation}
The infinitesimal representation of $\pi$ extends to the universal enveloping Lie algebra of $\mathfrak g$ that we identify with the left invariant differential operators on $G$.
Then for such a differential operator~$T$ we have 
$\mathcal F(Tf)(\pi) = \pi(T) \mathcal Ff (\pi)$
and we may write $\pi(T)={\mathcal F}(T)$. 
For instance, if as before ${\mathcal X}^{\alpha}$ denotes a product of $|\alpha|$  left invariant vector fields forming a basis of~${\mathfrak g}$, then 
$$\mathcal F({\mathcal X}^\alpha f)(\pi) = \pi({\mathcal X})^\alpha \mathcal Ff (\pi)\;\;{\rm and}\;\;
\mathcal F({\mathcal X}^\alpha ) = \mathcal F({\mathcal X})^\alpha.$$ 
Note that $\pi({\mathcal X})^\alpha$ may be considered as a field of unbounded operators on $\widehat G$ defined on the smooth vectors of the representations \cite{FR}.

\subsection{The sublaplacian} \label{freq}
The sublaplacian on $G$ is defined by
$$
\Delta_{G}:= \sum_{j=1}^{2d} V_j^2. 
$$
One checks easily that $\Delta_G$ is a differential operator which is left invariant and homogeneous of degree $2$.
In this paper, we shall consider its associated
{\it Schr\"odinger equation}
\begin{equation*}\label{eq:schro}
i\partial_t \psi  = - {1\over 2}\Delta_G \psi,\;\; \psi_{t=0}=\psi _0.
\end{equation*}
The operator $\Delta_G$ is essentially self-adjoint  on $C_c^\infty(G)$ (see \cite[Section 4.1.3]{FR} or \cite[Proof of Lemma 12.1]{strichartz}), 
so the Schr\"odinger equation has a unique solution for any data $\psi_0\in L^2(G)$ by Stone's theorem.  We keep the same notation for its unique self-adjoint (unbounded) extension to $L^2(G)$. 
More precisely, to deal with high-frequencies data, we shall be concerned with the {\it semi-classical  Schr\"odinger equation}
\begin{equation}\label{eq:schrosc}
i\eps^\tau \partial_t \psi^\eps = -{\eps^2\over 2} \Delta_G \psi^\eps,\;\;(\psi^\eps )_{|t=0}=\psi^\eps_0,
\end{equation} 
where $\eps>0$ is a small parameter taking into account the size of the oscillations of the initial data and $\tau>0$ is a parameter allowing us to consider large time behaviour, and as the same time the asymptotics $\eps\rightarrow 0$. 

\medskip 

The definition of $\Delta_G$ is independent of the chosen orthonormal basis for $\mathfrak v$ - although it depends on the scalar product that we have fixed at the very beginning on ${\mathfrak v}$. 
In particular, choosing the basis fixed in Section \ref{subsec_preliminaries}  for any $\lambda\in \mathfrak z^* \setminus\{0\}$
we have
\begin{equation}
\label{eq_DeltaGPQ}	
 \Delta_{G}= \sum_{j = 1}^{d} (P_j^2 + Q_j^2).
 \end{equation}

The  infinitesimal representation (or Fourier transform) of $\Delta_G$  can be computed thanks to 
the equalities in~\eqref{eq_piPQZ}, \eqref{Fourieromega} 
 and~\eqref{eq_DeltaGPQ}: 
at $\pi^{(0,\omega)}$, $\omega\in{\mathfrak v}^*$, it is the number
$$
{\mathcal F} (-\Delta_G)(0,\omega) = |\omega|^2,
$$
and  at $\pi^\lambda$, $\lambda\in{\mathfrak z}^*\setminus\{0\}$, it is the operator 
\begin{equation}\label{def:H}
{\mathcal F} (-\Delta_G)(\lambda) = H(\lambda),
\end{equation}
where~$ H(\lambda)$ is defined on
$L^2(\R^d)$ through the identification $\mathfrak p_\lambda\sim \R^d$, by
\begin{equation}\label{def:H2}
H(\lambda )=|\lambda| \sum_{1\leq j\leq d} \left( -\partial_{\xi_j}^2+\xi_j^2\right).
\end{equation}
Up to a constant, this is the quantum harmonic oscillator.
The spectrum $\{|\lambda|(2n+d), n\in \N\}$ of $H(\lambda)$ is discrete
and the eigenspaces are finite dimensional.
To each eigenvalue $|\lambda|(2n+d)$, we denote by 
$
\Pi_n^{(\lambda)}$ and 
$\mathcal V_n^{(\lambda)} 
$
 the corresponding spectral orthogonal projection and eigenspace.
The well-known description of the eigenspaces in terms of Hermite functions is recalled in Appendix \ref{sec_Hermite+pflem_sigma_comH}.
In particular, for each $n\in \N$, all the eigenspaces $\mathcal V_n^{(\lambda)}$, $\lambda\in \mathfrak z^*\setminus\{0\}$, are isomorphic, and may be denoted by $\mathcal V_n$.

\subsection{Semi-classical pseudodifferential operators}

\subsubsection{The space ${\mathcal A}_0$ of semi-classical symbols}
We denote by $\mathcal A_0$ the space of symbols 
$\sigma = \{\sigma(x,\pi) : (x,\pi)\in G\times \widehat G\}$ of the form 
$$
\sigma(x,\lambda)=\mathcal F \kappa_x (\lambda) = \int_G \kappa_x(z) (\pi^\lambda_z)^* dz, 
$$
where $x\mapsto \kappa_x(\cdot)$ is a smooth and compactly supported function from $G$ to $\mathcal S(G)$.
Being compactly supported means that $\kappa_x(z)=0$ for $x$ outside a compact of $G$ and any $z\in G$.

\begin{remark}\label{rem:linkSinfty}
The algebra $\mathcal A_0$ is the space of smoothing symbols with compact support in $x$.
   We will recall the definition of the space $S^{-\infty}$ of smoothing symbols introduced in~\cite{FR} at the beginning of Section~\ref{subsec:preliminaries} below. 
   Examples of smoothing symbols include the spectrally defined symbols $f(H(\lambda))$ for any $f\in \mathcal S (\R)$~\cite[Chapter 4]{FR}. 
\end{remark}

As the Fourier transform is injective, it yields a one-to-one correspondence between the symbol~$\sigma$ and the function $\kappa$:
we have $\sigma(x,\lambda)=\mathcal F \kappa_x(\lambda)$ and conversely the Fourier inversion formula \eqref{inversionformula}
yields 
$$
\forall x,z\in G,\;\; \kappa_x(z)= c_0\int_{\widehat G} {\rm Tr} \left( \pi^{\lambda}_ {z} \sigma(x,\lambda)\right) |\lambda|^d d\lambda.$$
The set $\mathcal A_0$ is an algebra for the composition of symbols since 
if $\sigma_1(x,\lambda)=\mathcal F \kappa_{1,x} (\lambda)$
and $\sigma_2(x,\lambda)=\mathcal F \kappa_{2,x} (\lambda)$
are in $\mathcal A_0$, 
then so is $\sigma_1(x,\lambda)\sigma_2(x,\lambda)
=\mathcal F (\kappa_{2,x}*\kappa_{1,x}) (\lambda)$
by \eqref{fourconv}.

\medskip 

 In
 the case of representations of finite dimension,  we  distinguish between all the finite dimensional representations by replacing $\lambda=0$ with the parameters $(0,\omega)$, $\omega\in \mathfrak v^*$. 
 The  operator $\mathcal F \kappa_x(0,\omega)= \sigma(x,(0,\omega))$ then reduces to a complex number since ${\mathcal H}_{(0,\mu)}=\C$.

\subsubsection{Semi-classical pseudodifferential operators} 
Given $\eps>0$, the {\it semi-classical parameter}, that we use to
weigh the oscillations of the functions that we shall consider, 
we quantify the symbols of ${\mathcal A}_0 $ by setting  as in~\cite{FF} (see also \cite{FR, BFG1, Taylor84})

\begin{equation}
	\label{eq_quantization}
	{\rm Op}_\eps(\sigma) f(x)= c_0 \int_{\widehat G} {\rm Tr} \left(  \pi^{\lambda}_{x} \sigma(x,\eps^2\lambda) {\mathcal F} f(\lambda)  \right)|\lambda|^d \,d\lambda,\;\; f\in{\mathcal S}(G).
\end{equation}
The kernel of the operator ${\rm Op}_\eps(\sigma)$ is the function 
$$
G\times G\ni (x,y)\mapsto \kappa^\eps_x(y^{-1} x)
$$ where $\kappa^\eps_x(z)=\eps^{-Q} \kappa_x\left(\delta_{\eps^{-1}} z\right)$ 
and $\kappa_x $, which is such that $\mathcal F(\kappa_x)(\lambda)=\sigma(x,\lambda)$, is called 
the {\it convolution kernel} of $\sigma$. Note that $\eps^2 \lambda$ can be understood as the action on $\lambda$ of the dilation induced on~$\widehat G$ by the dilation $\delta_\eps$ of $G$ (see Remark~3.3 in~\cite{FF}). 

\medskip 

Following~\cite{FF}, 
 the 
action of the symbols in ${\mathcal A_0}$  on $L^2(G)$ is bounded:
$$
\exists C>0, \;\;
\forall \sigma\in {\mathcal A}_0,\;\;
\forall \eps>0,\;\;\| {\rm Op}_\eps(\sigma)\|_{{\mathcal L}(L^2(G))}\leq C \, \int_G \sup_{x\in G} |\kappa_x (z)|dz.
$$

One also has to mention that there exists a symbolic calculus for these operators (see~\cite{FF2}). In this paper, we will mainly use the description of the commutator between the sub-Laplacian and a semi-classical pseudodifferential operator, which comes from the explicit computation and writes: 
for all $\sigma\in{\mathcal A}_0$,
\begin{equation}\label{prop:commutatorDelta}
[ -\eps^2 \Delta_G ,{\rm Op}_\eps (\sigma)] 
=
  {\rm Op}_\eps \left([H(\lambda),\sigma]\right)
\ - \ 
2\varepsilon  \, {\rm Op}_\eps \left(V\cdot \pi^\lambda(V) \sigma\right)
\ - \ 
\eps^2 {\rm Op}_\eps \left(\Delta_G \sigma\right),
\end{equation}
 where $H(\lambda)= {\mathcal F} (-\Delta_G)$ has been defined in~(\ref{def:H}) and~(\ref{def:H2}).

 \subsubsection{The subspace ${\mathcal A}_H$ of ${\mathcal A}_0$}

The Egorov Theorem that we are going to state in the next section is valid for  symbols compactly supported with respect to both the Fourier transform and the spectral decomposition of $H(\lambda)$:

\begin{definition}
\label{def_AH}
A symbol $\sigma\in {\mathcal A}_0$ is in ${\mathcal A}_H$ when it vanishes for $\lambda$ in a neighbourhood  of $\{\lambda=0\}$  and when
its kernel and image contain a finite number of $\mathcal V_n$, 
in the sense that we have
$$
\forall (x,\lambda)\in  G\times (\mathfrak z^*	\setminus\{0\})
\qquad \Pi_n^{(\lambda)} \sigma(x,\lambda) \, \Pi_{n'}^{(\lambda)}= 0,
$$
for all but a finite number of integers $n,n'\in \N$.
\end{definition}

One checks readily that  ${\mathcal A}_H$ is a subalgebra of ${\mathcal A}_0$. 
It is non-trivial since 
it contains for instance all the symbols of the form $a(x) f(H(\lambda)) b(\lambda)$ where $a\in{\mathcal C}_c^\infty(G)$, $f\in {\mathcal C}_c^\infty(\R)$ and where $b\in{\mathcal S} (\mathfrak z^*)$ vanish in a neighbourhood of 0 (see Remark \ref{rem:linkSinfty}); 
for other symbols in $\mathcal A_H$,  see Remark \ref{rem_f(H)}.
Although ${\mathcal A}_H$  cannot be dense in ${\mathcal A}_0$ for the Fr\'echet topology of ${\mathcal A}_0$, 
we will see in  Corollary \ref{cor_cv_AH0} that it satisfies a property of weak density. 
Besides, 
symbols $\sigma\in \mathcal A_H$ can be decomposed in commuting and non-commuting symbols according to the following definition. 

\begin{definition}
With the setting of Defintion \ref{def_AH}, the symbol $\sigma\in {\mathcal A}_H$ is called $H$-{\it diagonal} when 
$\sigma(x,\lambda)=\sum_{n\in\N} \Pi_n^{(\lambda)}\sigma(x,\lambda)\, \Pi_n^{(\lambda)}$  and {\it anti-$H$-diagonal} when $\sigma(x,\lambda)=\sum_{n\not=n'} \Pi_n^{(\lambda)}\sigma(x,\lambda)\, \Pi_{n'}^{(\lambda)}$. We  denote by $\mathcal A_H^{(d)}$ the space of $H$-diagonal symbols.
\end{definition}

\begin{lemma}
\label{lem_sigmad+a}
A symbol  $\sigma\in \mathcal A_H$ is uniquely decomposed as the sum 
$\sigma = \sigma^{(d)} + \sigma^{(a)}$
of  an $H$-diagonal symbol $\sigma^{(d)}	\in \mathcal A_H$ with an anti-$H$-diagonal symbol $\sigma^{(a)}\in \mathcal A_H$.
Furthermore, for each $n,n'\in \N$, the symbol 
$\Pi_n\sigma \Pi_{n'}$ given by  
$$
(\Pi_n\sigma\Pi_{n'})(x,\lambda)
= 
\left\{\begin{array}{ll}
	\Pi_n^{(\lambda)}\sigma(x,\lambda)\, \Pi_{n'}^{(\lambda)} \ \mbox{for}\ (x,\lambda)\in G \times \mathfrak z^*\setminus\{0 \},
	\\
	0 \ \mbox{for}\ x\in G \ \mbox{and} \ \lambda=0,
\end{array}\right.
$$
is in $\mathcal A_H$.
\end{lemma}

Lemma \ref{lem_sigmad+a} will be a consequence of Corollary \ref{cor_prop_PiindotS}, see Remark \ref{rem_pf_lem_sigmad+a}.

\subsection{The Egorov Theorem on H-type groups} 
 For $s\in \R$, we define the flow $\Psi^s$ on $G\times (\mathfrak z^*\setminus\{0\})$ via
\begin{equation}
\label{eq_Psis}	
\Psi^s: 
\left\{\begin{array}{rcl}
	G\times (\mathfrak z^*\setminus\{0\})&\longrightarrow & G\times (\mathfrak z^*\setminus\{0\})\\
(x,\lambda) &\longmapsto &
\left({\rm Exp} (s {\mathcal Z}^{(\lambda)} ) x,\lambda\right)
\end{array}\right.
.
\end{equation}
In particular, this map may be composed with symbols with support in $G\times (\mathfrak z^*\setminus\{0\})$ such as the symbols in $\mathcal A_H$.
This and Lemma \ref{lem_sigmad+a} allows us to define the following action on $\mathcal A_H^{(d)}$:
\begin{equation}
\label{eq_Phis}
\Phi^s: 
\left\{\begin{array}{rcl}
	\mathcal A_H^{(d)} &\longrightarrow & \mathcal A_H^{(d)} \\
\sigma = \sum_{n} \Pi_n \sigma \Pi_{n} &\longmapsto &
\Phi^s(\sigma)= \sum_{n} (\Pi_n \sigma \Pi_{n}) \circ \Psi^{\frac {2n+d}{2|\lambda|} s}
\end{array}\right. , \quad s\in \R.
\end{equation}

\begin{theorem}\label{theo:Egorov}
Let $(\psi^\eps_0)_{\eps>0}$ be a bounded family in $L^2(G)$
and $\psi^\eps(t) = {\rm e}^{i \eps^{2-\tau} t \Delta_G} \psi^\eps_0$ be the solution to  \eqref{eq:schrosc}.
Let $\sigma= \sigma^{(d)} + \sigma^{(a)}\in \mathcal A_H$ be decomposed into $H$-diagonal and anti-$H$-diagonal 
as in Lemma \ref{lem_sigmad+a}.
Let  $\theta\in{\mathcal C}^\infty_c(\R)$.
\begin{itemize}
	 \item[(i)]  For the anti-$H$-diagonal part, we have
$$
\int _{\R} \theta(t) \left( {\rm Op}_\eps (\sigma^{(a)}) \psi^\eps(t),\psi^\eps(t)\right)_{L^2(G)} dt=O(\eps^{\min(\tau,1)}).
$$

\item[(ii)] For the $H$-diagonal part, 
we have the following alternative:
\begin{enumerate}
\item if $\tau\in(0,2)$, 
$$
\int_{\R} \theta'(t) \left({\rm Op}_\eps(\sigma^{(d)} )\psi^\eps(t),\psi^\eps(t)  \right) dt =O(\eps^{\min(1,2-\tau)}),
$$
\item if  $\tau=2$, for all $s\in\R$ (transport)
$$
\int_{\R} \theta(t) \left({\rm Op}_\eps(\sigma^{(d)} )\psi^\eps(t),\psi^\eps(t)  \right) dt=
\int_{\R} \theta(t+s) 
\left({\rm Op}_\eps (\Phi^{-s}(\sigma^{(d)}) ) \psi^\eps(t),\psi^\eps(t)  \right) dt  +O(\eps),
$$
\item if $\tau>2$, for all $s\in\R$ (invariance) 
$$\int_{\R} \theta(t) \left({\rm Op}_\eps(\sigma^{(d)} )\psi^\eps(t),\psi^\eps(t)  \right) dt=
\int_{\R} \theta(t) \left({\rm Op}_\eps(\Phi^{-s}(\sigma^{(d)}) )\psi^\eps(t),\psi^\eps(t)  \right) dt  +O(\eps^{\min(1,\tau-2)}).$$
\end{enumerate}
In Parts (2) and (3), we use the action $\Phi^s$ defined in \eqref{eq_Phis}.
\end{itemize}
\end{theorem}

It may appear  unusual to have an Egorov Theorem holding in the space of distributions in the time variable. However, it is already the case in the Euclidean case when one works with the propagator of a Schr\"odinger operator with matrix-valued potential $-{\eps^2\over 2} \Delta \, {\rm Id}  + V(x)$ with $V$ matrix-valued (see~\cite{GMMP,FL1,FL2}). 
The proof of this Theorem is postponed until Section~\ref{sec:proofs}.

\subsection{Time averaged semi-classical measures and the quantum limits}\label{sec:semiclas}

We now want to pass to the limit $\eps\rightarrow 0$ in expressions of the form~(\ref{eq:egorov}) and to identify the limiting objects, together with their properties.  
For this purpose, we follow \cite[Section 5]{FF} with slightly different notation and introduce the following vocabulary for operator valued measures:

\begin{definition}
\label{def_gammaGamma}
	Let $Z$ be a complete separable metric space, 
	and let $\xi\mapsto {\mathcal H}_\xi$ a measurable field of complex Hilbert spaces of $Z$.
\begin{itemize}
\item 
	The set 
	$ \widetilde{\mathcal M}_{ov}(Z,({\mathcal H}_\xi)_{\xi\in Z})$
	is the set of pairs $(\gamma,\Gamma)$ where $\gamma$ is a positive Radon measure on~$Z$ 
	and $\Gamma=\{\Gamma(\xi)\in {\mathcal L}({\mathcal H}_\xi):\xi \in Z\}$ is a measurable field of trace-class operators
such that
$$\|\Gamma d \gamma\|_{\mathcal M}:=\int_Z{\rm Tr}_{{\mathcal H}_\xi} |\Gamma(\xi)|d\gamma(\xi)
<\infty.$$

 \item 
	Two pairs $(\gamma,\Gamma)$ and $(\gamma',\Gamma')$ 
in $\widetilde {\mathcal M}_{ov}(Z,({\mathcal H}_\xi)_{\xi\in Z})$
are {equivalent} when there exists a measurable function $f:Z\to \mathbb C\setminus\{0\}$ such that 
$$d\gamma'(\xi) =f(\xi)  d\gamma(\xi)\;\;{\rm  and} \;\;\Gamma'(\xi)=\frac 1 {f(\xi)} \Gamma(\xi)$$ for $\gamma$-almost every $\xi\in Z$.
The equivalence class of $(\gamma,\Gamma)$ is denoted by $\Gamma d \gamma$,
and the resulting quotient set is 
 denoted by ${\mathcal M}_{ov}(Z,({\mathcal H}_\xi)_{\xi\in Z})$.

\item 
A pair $(\gamma,\Gamma)$ 
in $ \widetilde {\mathcal M}_{ov}(Z,({\mathcal H}_\xi)_{\xi\in Z})$
 is {positive} when 
$\Gamma(\xi)\geq 0$ for $\gamma$-almost all $\xi\in Z$.
In  this case, we may write  $(\gamma,\Gamma)\in  \widetilde {\mathcal M}_{ov}^+(Z,({\mathcal H}_\xi)_{\xi\in Z})$, 
and $\Gamma d\gamma \geq 0$ for $\Gamma d\gamma \in {\mathcal M}_{ov}^+(Z,({\mathcal H}_\xi)_{\xi\in Z})$.
\end{itemize}
\end{definition}

By convention and if not otherwise specified,  a representative of the class $\Gamma d\gamma$ is chosen such that ${\rm Tr}_{{\mathcal H}_\xi} \Gamma=1$. 
In particular, if ${\mathcal H}_\xi$ is $1$-dimensional, $\Gamma=1$ and  $\Gamma d\gamma$ reduces to the measure $d\gamma$.
One checks readily that $\mathcal M_{ov} (Z,({\mathcal H}_\xi)_{\xi\in Z})$ equipped with the norm $\| \cdot\|_{{\mathcal M}}$ is a Banach space.

\medskip 

When the field of Hilbert spaces is clear from the setting, 
we may write 
$$
\mathcal M_{ov} (Z) = \mathcal M (Z,({\mathcal H}_\xi)_{\xi\in Z}),
\quad
\mbox{and}\quad
\mathcal M_{ov}^+ (Z) = \mathcal M^+ (Z,({\mathcal H}_\xi)_{\xi\in Z}),
$$
for short.
For instance, if $\xi\mapsto {\mathcal H}_\xi$ is given by $\mathcal H_\xi=\mathbb C$ for all $\xi$, 
then $\mathcal M (Z)$ coincides with the space of finite Radon measures on $Z$.
Another example is when $Z$ is of the form $Z=Z_1 \times \widehat G$ where $Z_1$ is a complete separable metric space, and $\mathcal H_{(z_1,\lambda)}= {\mathcal H}_\lambda$, where
 the Hilbert space ${\mathcal H}_\lambda$ is associated with the representation of $\lambda \in \widehat G$
(that is, using the description in \eqref{eq_widehatG},
${\mathcal H}_\lambda$ is equivalent to $L^2(\mathfrak p_\lambda)$ if the representation corresponds to $\lambda \in \mathfrak z^*\setminus\{0\}$ and ${\mathcal H}_{(0,\omega)}=\C$ if $\lambda=0$ and the representation corresponds to $(0,\omega)$ with $\omega\in{\mathfrak v}^*$). 

\medskip

We will often consider measurable bounded maps of the time variable, valued in the space of measures that are positive as scalar-valued or operator-valued measures:
\begin{definition}
\label{def_LinftyRX+}	
If $X$ denotes the Banach space $\mathcal M(Z)$ or more generally 
$\mathcal M_{ov}(Z)$  as in  Definition \ref{def_gammaGamma}, then 
 $L^\infty (\R, X^+)$ denotes the space of maps of $t\in \R$ and valued in $X$ in $L^\infty(\R,X)$ with positive values for almost every $t\in \R$. 
\end{definition}

\subsubsection{Time-averaged semi-classical measures}
\label{subsec_timeav_scm}

With a bounded family $(u^\eps)_{\eps>0}$ in $L^\infty(\R, L^2(G))$, we associate  the quantities 
 \begin{equation}
 \label{def:leps}
\ell_\eps(\theta, \sigma)=\int_{\R} \theta(t)  \left({\rm Op}_\eps (\sigma) u^\eps(t),u^\eps(t)\right)_{L^2(G)} dt,\;\;\sigma\in \mathcal A_0,\;\;\theta\in L^1(\R),
\end{equation}
the limits of which are characterized by a map in $L^\infty(\R,{\mathcal M}_{ov}^+(G\times \widehat G))$.

\begin{theorem}\label{theo:mesures}
Let $(u^\eps)_{\eps>0}$ be a bounded family in $L^\infty(\R,L^2(G))$.
There exist a sequence $(\eps_k)_{k\in \mathbb N}$ in $(0,+\infty)$ with  $\eps_k\Tend{k}{+\infty}0$
and a  map  $t\mapsto \Gamma_t d\gamma_t $ in
$L^\infty(\R, {\mathcal M}_{ov}^+(G\times \widehat G))$
 such that we have for all $ \theta\in L^1(\R)$ and $\sigma\in {\mathcal A}$,
$$
\int_{\R} \theta(t) \left({\rm Op}_{\eps_k} (\sigma) u^{\eps_k}(t),u^{\eps_k}(t)\right)_{L^2(G)}dt 
\Tend {k}{+\infty} 
\int_{\R\times G\times \widehat G} \theta(t){\rm Tr}\left(\sigma(x,\lambda) \Gamma_t(x,\lambda)\right)d\gamma_t(x,\lambda) dt.
$$
Given the sequence $(\eps_k)_{k\in \mathbb N}$, 
the map $t\mapsto \Gamma_t d\gamma_t$  
is unique up to equivalence. Besides,
$$
\int_{\R}\int_{G\times \widehat G} 
{\rm Tr}\left( \Gamma_t(x,\lambda) \right) d\gamma_t(x,\lambda) \, dt \leq  \limsup_{\eps\rightarrow 0}\| u^\eps\|_{L^\infty(\R,L^2(G))}.
$$
\end{theorem}

We call
 the map $t\mapsto\Gamma_t d\gamma_t$  satisfying Theorem~\ref{theo:mesures} (for some subsequence $\eps_k$) a {\it time-averaged semi-classical measure} of the family $(u^\eps(t))$. 
 Note that we have not assumed any estimate of the form~(\ref{Hscriterium}) on the family $u^\eps$ in order to define its time averaged semi-classical measure; 
 such additional property will however be useful to determine the limits of the time-averaged densities associated with $u^\eps$ in terms of time-averaged semi-classical measures, as we shall see in Section~\ref{sec:densitylimit}.

   \begin{remark}
   \begin{enumerate}
\item   Note that this result can be generalised to any graded Lie group:
for any bounded family
$(u^\eps)_{\eps>0}$ in $L^\infty(\R,L^2(G))$,
there exist a sequence $(\eps_k)_{k\in \mathbb N}$ in $(0,+\infty)$ with  $\eps_k\Tend{k}{+\infty}0$
and a map   $t\mapsto \Gamma_t d\gamma_t$ in $L^\infty(\R, {\mathcal M}_{ov}^+(G\times \widehat G))$
 such that we have
 $$ 
 \int_{\R}\theta(t) \left({\rm Op}_{\eps_k} (\sigma) u^{\eps_k}(t),u^{\eps_k}(t)\right)_{L^2(G)}dt\Tend {k}{+\infty} \int_{\R\times G\times \widehat G} \theta(t){\rm Tr}\left(\sigma(x,\pi) \Gamma_t(x,\pi)\right)d\gamma_t(x,\pi)dt,
 $$
  for every $ \theta\in{\mathcal C}_c^\infty(\R)$ and $\sigma\in \mathcal A_0$.
\item In the case of this article where $G$ is H-type, the special structure of $\widehat G$ implies that $\Gamma_t d\gamma_t$ consists of two pieces, one localized above $\lambda\in\mathfrak z^*\setminus\{0\}$ and another one which is scalar above~$\mathfrak v^*$, see \eqref{eq_widehatG}.
\end{enumerate}
   \end{remark}

\subsubsection{Semi-classical measures and Schr\"odinger equation}\label{sec:enoncetheo}
Our main theorem regarding semi-classical measures of solutions to the Schr\"odinger equation is the following:

\begin{theorem}
\label{theo:schro} 
Let  $(\psi_0^\eps)_{\eps>0}$ be a 
 bounded family in $L^2(G)$ 
 and $\psi^\eps(t) = {\rm e}^{i \eps^{2-\tau} t \Delta_G} \psi^\eps_0$ be the solution to  \eqref{eq:schrosc}.
Then any semi-classical measure 
$t\mapsto \Gamma_t d\gamma_t\in L^\infty(\R, {\mathcal M}_{ov}^+(G\times \widehat G))$
as in Theorem \ref{theo:mesures} for the family $u^\eps (t) =\psi^\eps(t)$, 
satisfies the following additional properties:
\begin{itemize}
\item[(i)] 
For almost every $(t,x,\lambda)\in \R\times G\times \widehat G$, the operator $\Gamma_t(x,\lambda)$ commutes with $\widehat \Delta_G(\lambda)=H(\lambda)$:
\begin{equation}\label{eq:decomp}
\Gamma_t(x,\lambda)=\sum_{n\in\N } \Gamma_{n,t}(x,\lambda)\;\;{ with}\;\; \Gamma_{n,t}(x,\lambda):= \Pi_n^{(\lambda)}\Gamma_t(x,\lambda) \Pi_n^{(\lambda)},
\end{equation}
where $\Pi_n$ is the homogeneous symbol given by the spectral projection of $H(\lambda)$ for the eigenvalue $|\lambda|(2n+d)$
(see Section~\ref{freq}).
\item[(ii)] 
For each $n\in \N$,
the map $(t,x,\lambda)\mapsto \Gamma_{n,t}(x,\lambda) d\gamma_t(x,\lambda)$ defines
 a distribution on $\R\times G\times (\mathfrak z^*\setminus\{0\})$ valued in the finite dimensional space ${\mathcal L}(\mathcal V_n)$ which 
satisfies the following alternatives:
 \begin{enumerate}
\item if $\tau\in (0,2)$, 
$$
\partial_t\left(\Gamma_{n,t}(x,\lambda) d\gamma_t(x,\lambda)\right)=0,
$$
\item if $\tau=2$, 
$$
\left(\partial_t -{2n+d\over 2|\lambda|}  {\mathcal Z}^{(\lambda)}\right)\left(\Gamma_{n,t}(x,\lambda)d\gamma_t(x,\lambda)\right)=0
$$
where ${\mathcal Z}^{(\lambda)}\in {\mathfrak z}$ is the vector corresponding to $\lambda$ (see Section \ref{subsubsec_ONBg}),
\item if $\tau>2$, then the distribution
$\Gamma_{n,t} d\gamma_t$ is invariant under the flow of the vector field  $\mathcal Z^{(\lambda)}$ and thus is equal to~$0$. 
\end{enumerate}

\item[(iii)] Above $\lambda=0$, 
the map $t \mapsto \varsigma_t $ in $L^\infty (\R, \mathcal M^+ (G\times \mathfrak v^*))$
defined via
$$
d\varsigma_t(x,\omega)= \Gamma_t(x,(0,\omega))d\gamma_t(x,(0,\omega)){\bf 1}_{\lambda=0},
$$ 
 satisfies the following alternatives:
\begin{enumerate}
\item if $\tau\in(0,1)$, the map
$t\mapsto \varsigma_t$ is constant from $\R$ to $\mathcal M^+ (G\times \mathfrak v^*)$,
\item if $\tau=1$, 
then the map
$t\mapsto \varsigma_t$ is weakly continuous from $\R$ to $\mathcal M (G\times \mathfrak v^*)$,
and for all $t\in \R$
$$
\varsigma_t (x,\omega) = 
\varsigma_0 \left({\rm Exp} (t\,  \omega\cdot  V ) x,\omega\right), 
$$
where $\omega\cdot V= \sum_{j=1}^d \omega_j V_j \in \mathfrak g$,
\item if $\tau>1$, 
the measures
$\varsigma_t$  are invariant under the flow of the vector field $\omega\cdot V$ and thus
 supported on $G\times \{\omega=0\}$. 
\end{enumerate}
\end{itemize}
\end{theorem}

The existence of the semi-classical measure $\Gamma_t d\gamma_t$ follows from Theorem \ref{theo:mesures}, 
while  its additional properties come from the fact that $\psi^\eps(t)$ solves the Schr\"odinger equation. 
Point (i) of Theorem~\ref{theo:mesures} is a consequence of (i) of Theorem~\ref{theo:Egorov}. 
It will then appear that we will need to only use symbols which commute with $H(\lambda)$.

\medskip 

Before closing this section, we discuss why  the invariance of a semi-classical measure by vector fields imply that it is~$0$ (as in (3) of (ii)) or that its support has special properties (as in (3) of (iii)).
\begin{proof}[Proof of Part (3) in (ii)  and Part (3) in (iii)]The invariance properties have consequences that have been already studied in the Euclidean case in \cite[Lemma~3.6]{CFM}.  We adapt them to the setting of (3) of (ii) in the following way. 
First let $N_{\widehat G} :\widehat G \to [0,+\infty)$ be defined via
$$
N_{\widehat G}
(\lambda):=\left\{\begin{array}{l}
\sqrt{|\lambda|}
\;\;  \mbox{if}\;\;  
\lambda\not=0\\
|\omega|
\;\;\mbox{if}\;\; \lambda=(0,\omega),\;\;\omega\in{\mathfrak v}^*;
\end{array}\right.
$$
by \cite[Section 2.3]{FF}, it is continuous.
We also define the usual quasi-norm on $G$ via 
$$
|{\rm Exp} (V+Z)|=(|X|^4+|Z|^2)^{1/4}, 
\qquad
 \ \mbox{with} \ 
V\in {\mathfrak v} \ \mbox{and}\  Z\in{\mathfrak z}.
$$
We can now define the continuous function $N:G\times \widehat G \to [0,+\infty)$ with
$$
N(x,\lambda) =  (|x|^4+N_{\widehat G}
(\lambda)^4)^{1/4}.
$$
Let $x={\rm Exp} (V+Z)\in G $ with $V\in {\mathfrak v}$ and $z\in{\mathfrak z}$. In view of 
$${\rm Exp}(s\,  {\mathcal Z}^{(\lambda)} )x=  {\rm Exp} (V+Z+s{\mathcal Z}^{(\lambda)})$$
we deduce 
$$
|{\rm Exp}(s\,  {\mathcal Z}^{(\lambda)} )x|^4 = |V|^4+ |Z+s{\mathcal Z}^{(\lambda)}|^2 + N_{\widehat G}
(\lambda)^4.
$$ As a consequence, if
 $K$ is any compact subset of $G\times ({\mathfrak z}^*\setminus\{0\})$, then there exist constants $\alpha_1,\beta_1,\alpha_2,\beta_2,s_0>0$ such that $$
\forall s\geq s_0, \qquad
\forall (x,\lambda)\in K\qquad
\alpha_1 |s|^{1/2} -\beta_1 < |{\rm Exp}(s\,  {\mathcal Z}^{(\lambda)} )x | < \alpha_2|s|^{1/2}+\beta_2,$$
which is enough for the proof of  Lemma~3.6 in~\cite{CFM}. The measure  ${\rm Tr} (\Gamma_t) d\gamma_t$ which is invariant under the flow $\Psi^s$ is $0$ above~$K$. \\
A similar argument can be performed for (3) of (iii) since the only invariant set  by the action of~$\Xi^s$  is 
the set $G\times \{0\}$
and
 if $K$ is a compact subset of $G\times {\mathfrak v}^*$  such that $K \cap (G\times \{0\})=\emptyset$, then there exists $\alpha_1,\beta_1,\alpha_2,\beta_2,s_0 >0$ such that  for $s\geq _0$,
$$\alpha_2 |s| -\beta_2 < |{\rm Exp}(s\,  \omega\cdot V )x ) | < \alpha_2|s|+\beta_2.$$
Therefore, the measure $\varsigma_t(x,\omega)$ is supported on $G\times \{\omega=0\}$.
\end{proof}


\section{The $C^*$-algebra ${\mathcal A}$ associated with semi-classical symbols} \label{sec:3.1}

In this section we introduce the $C^*$-algebra formalism which can be associated with semi-classical symbols. The properties of this algebra, introduced in Section~\ref{sec:defA} are at the roots of our analysis and allow us to prove Theorem~\ref{theo:mesures} in Section~\ref{sec:prooftheomesures}. 
The proofs of Theorems \ref{theo:Egorov} and \ref{theo:schro} will use several ingredients. 
First,  it 
requires the analysis of the symbolic properties of the eigenprojectors~$\Pi_n^{(\lambda)}$ performed in Section~\ref{subsec:preliminaries}. 
Then,
in order to pass to the limits in the relations of the Egorov theorem~\ref{theo:Egorov}, we will need to approximate general symbols in ${\mathcal A}_0$ by symbols  belonging to the class ${\mathcal A}_H$, which is done in Section~\ref{subsec_approxPinproj}.
Finally, it will use symbols that commute with~$H$, the space of which is studied in Section~\ref{sec:Hdiag}.

\subsection{The $C^*$-algebra ${\mathcal A}$ and its states} \label{sec:defA}
 We introduce the algebra~${\mathcal A}$ which is  the closure of~${\mathcal A}_0$ for the norm
$\| \cdot\|_{\mathcal A}$ given by
\begin{equation}
 \label{eq_normA}	
  \| \sigma\|_{\mathcal A}:= 
  \sup_{(x,\lambda)\in G\times \widehat G} 
  \|\sigma(x,\lambda)\|_{\mathcal L (L^2(\mathfrak p_\lambda))}.
\end{equation}
Clearly, $\mathcal A$ is a sub-$C^*$-algebra of the tensor product of the commutative $C^*$-algebra $C_0(G)$ of continuous functions on $G$ vanishing at infinity together with $L^\infty(\widehat G)$.

It turns out that one can identify its spectrum in the following way:

\begin{proposition}
\label{prop:C*A}
 The set ${\mathcal A}$ is a separable $C^*$-algebra of type~1.
It is not unital  but admits an approximation of identity. 
Besides, if $\pi_0\in \widehat G$ and $x_0\in G$, then the mapping 
$$
\left\{\begin{array}{lll}
{\mathcal A}_0 &\longrightarrow& {\mathcal L}({\mathcal H}_{\lambda_0})\\
\sigma&\longmapsto & \sigma(x_0,\pi_0)
\end{array}\right. \quad
$$
extends to a continuous mapping $\rho_{x_0,\pi_{0}}:{\mathcal A}\to {\mathcal L}({\mathcal H}_{\pi_0})$ which is an irreducible non-zero representation of ${\mathcal A}$. Furthermore, the mapping
$$
\left\{\begin{array}{lll}
 G\times \widehat G &\longrightarrow&\widehat{\mathcal A} \\
 (x_0, \pi_0) &\longmapsto& \rho_{x_0,\pi_0}
 \end{array}\right.
 $$
is a homeomorphism which allows for the identification of  $\widehat {\mathcal A}$ with $G\times \widehat G$.
\end{proposition}

The proof follows the lines of \cite[Section 5]{FF}. It utilises the fact  that, by definition,  the $C^*$ algebra $C^*(G)$ of the group $G$ is the closure of $\mathcal F \mathcal S(G)$ for $\sup_{\lambda\in \widehat G} \|\cdot\|_{\mathcal L(\mathcal H_\lambda)}$ and that the spectrum of $C^*(G)$ is $\widehat G$. This implies readily that $\mathcal A$ may be identified with the $C^*$-algebra of continuous functions which vanish at infinity on $G$ and are valued on $C^*(G)$ and that its spectrum is as described in Proposition \ref{prop:C*A}.
Furthermore, the algebraic span of the symbols of the form $\tau(x,\lambda)= a(x)b(\lambda)$ with $a(x)$ in $C_c^\infty(G)$ and the Fourier multiplier $b(\lambda)$ in $S^{-\infty}$ is dense in $\mathcal A$. Notice that their boundedness is easier to obtain since  ${\rm Op}_\eps(\tau)$ simply is the composition of the operator of multiplication by $a(x)$ and of the Fourier multiplier $b(\eps^2 \lambda)$, and one has 
\begin{equation}\label{simplebound}
\| {\rm Op}_\eps(\tau)\|_{{\mathcal L}(L^2(G))}\leq \sup_{x\in G,\,\lambda\in\widehat G} \| \tau \|_{{\mathcal L}(L^2({\mathfrak p}_\lambda))}.
\end{equation}

\medskip 

We can also describe the states of the $C^*$-algebra $\mathcal A$. 

\begin{proposition}
\label{prop:states}
 If $\ell$ is a  state of the $C^*$-algebra ${\mathcal A}$, then there exists  a pair $(\gamma,\Gamma)$ unique up to its equivalence class in
${\mathcal M}_{ov}^+(G\times \widehat G)$
which satisfies
  \begin{equation}
\label{eq_prop_trGgno1}
\int_{G\times \widehat G} {\rm Tr} \left(\Gamma(x, \lambda)\right) d\gamma(x, \lambda) =1,
\end{equation}
and 
\begin{equation}
\label{eq_prop_elltrGgno1}
\forall \sigma\in{\mathcal A}\qquad
\ell(\sigma) = \int_{G\times \widehat G} {\rm Tr}\left(\sigma(x,\lambda) \Gamma(x, \lambda)\right) d\gamma(x, \lambda).
\end{equation}
Conversely, if   
$\Gamma \, d\gamma\in {\mathcal M}_{ov}^+(G\times \widehat G)$ satisfies \eqref{eq_prop_trGgno1}, then the linear form $\ell$ defined via~\eqref{eq_prop_elltrGgno1} is a state of  ${\mathcal A}$. 
\end{proposition}

\begin{proof}
This proposition is a corollary of Proposition~4.1 in~\cite{FF2}. Its proof follows the lines of~\cite[Section 5]{FF} and  is performed in details in the Appendix of~\cite{FF2}.
\end{proof}

The description of the states of the $C^*$-algebra ${\mathcal A}$ yields the following corollary:
\begin{corollary}
\label{cor1_prop:states}
The topological dual ${\mathcal A}^*$ of 	 ${\mathcal A}$
may  be identified as a Banach space with 
${\mathcal M}_{ov}(G\times \widehat G)$ in the following way:
\begin{itemize}
\item 
If $\ell:{\mathcal A}\to \C$ is a continuous linear form, 
then there exists a unique element $\Gamma d\gamma\in {\mathcal M}_{ov} 	(G\times \widehat G)$ satisfying \eqref{eq_prop_elltrGgno1}. 
\item 
Conversely,  \eqref{eq_prop_elltrGgno1} defines a continuous linear form $\ell$ on ${\mathcal A}$.
\end{itemize}
Moreover, we have the following properties:
\begin{enumerate}
\item $\|\ell\|_{{\mathcal A}^*} = \|\Gamma d\gamma\|_{{\mathcal M}}$.
\item
$\Gamma d\gamma \geq 0$ if and only if $\ell(\sigma)\geq 0$ for every positive element $\sigma$ in the $C^*$-algebra ${\mathcal A}$. 
\item
$\ell$ is a state if and only if $	\Gamma d\gamma \geq 0$ and \eqref{eq_prop_trGgno1} holds.
\end{enumerate}
\end{corollary}

\subsection{Time dependent states and the proof of Theorem~\ref{theo:mesures}}\label{sec:prooftheomesures}
With this precise description of the topological dual ${\mathcal A}^*$, one can now consider functionals defined on ${\mathcal C}_c^\infty(\R)\times {\mathcal A}_0$.

\begin{proposition}
\label{cor2_prop:states}
Let $\ell: {\mathcal C}_c^\infty(\R)\times {\mathcal A}_0\to \C $ be a non-zero bilinear map satisfying 
\begin{equation}
\label{eq_cor_ellL1}
\forall \sigma\in {\mathcal A}_0,\quad
\forall \theta \in \mathcal C_c^\infty(\R),\qquad
\ell(\theta, \sigma)\leq \|\sigma\|_{\mathcal A}
\|\theta\|_{L^1(\R)},
\end{equation}
and 
\begin{equation}
\label{eq_cor_ellpos}
\forall \sigma\in {\mathcal A}_0,\quad
\forall \theta \in \mathcal C_c^\infty(\R),\qquad
\ell(|\theta|^2, \sigma^*\sigma)\geq 0,
\end{equation}
Then $\ell$ extends uniquely to a continuous bilinear map 
$\ell: L^1(\R)\times {\mathcal A}\to \C $ for which we keep the same notation.
Furthermore,  there exists a unique map $t\mapsto \Gamma_t d\gamma_t$ in $L^\infty(\R, {\mathcal M}_{ov}^+( G\times \widehat G))$  
satisfying $\|\Gamma_t d\gamma_t\|_{\mathcal M}=1$
for almost all $t\in \R$, and:
$$
\forall \sigma\in {\mathcal A},
\quad
\forall \theta \in L^1(\R),
\qquad
\ell(\theta,\sigma) 
= \int_{\R}\theta(t)\int_{G\times \widehat G}  {\rm Tr}\left(\sigma(x,\lambda) \Gamma_t(x, \lambda)\right) d\gamma_t(x, \lambda) \ dt.
$$
\end{proposition}

\begin{proof}
The estimate in \eqref{eq_cor_ellL1} implies that
the bilinear map $\ell$ extends uniquely into a continuous bilinear map on $L^1(\R)\times {\mathcal A}$ for which we keep the same notation.
Furthermore, 
for each $\sigma\in {\mathcal A}$, 
we identify the continuous linear map $\theta\mapsto \ell(\theta , \sigma)$ on the Banach space $L^1(\R)$ 
 with the function $\ell_\sigma\in L^\infty(\R)$ given via
$$
\forall \theta \in L^1(\R)
\qquad
\ell(\theta , \sigma) = \int_{\R} \theta(t) \ \ell_\sigma(t) \ dt.
$$
Note that  $\|\ell_\sigma\|_{L^\infty(\R)}\leq \|\sigma\|_{\mathcal A}$ and that the map $\sigma\mapsto \ell_\sigma$ is a linear mapping on ${\mathcal A}$ to $L^\infty(\R)$.
Hence, we can  view the map $L:t\mapsto (\sigma\mapsto \ell_\sigma(t))$ as a measurable 
bounded map from $\R$ to the Banach space  ${\mathcal A}^*$.
Moreover, 
the assumption in \eqref{eq_cor_ellpos} implies that 
$$
\forall \theta \in {\mathcal C}_c^\infty(\R)
\qquad
\int_{\R} |\theta(t)|^2 \ell_{\sigma^*\sigma}(t) dt
=
\ell (|\theta|^2 , \sigma^*\sigma)
\geq 0,
$$
hence 
 $\ell_{\sigma^*\sigma}(t)\geq 0$ for almost every $t\in \R$.
In other words, $\sigma\mapsto \ell_\sigma(t)$ is a state for almost every $t\in \R$.
Hence, the map $L\in L^\infty(\R, {\mathcal A}^*)$ 
is valued in the set of state of  ${\mathcal A}$.
Corollary \ref{cor1_prop:states}
with the identification of ${\mathcal A}^*$
 with ${\mathcal M}_{ov}(G\times \widehat G)$ 
allows us to conclude.
\end{proof}

With these concepts in mind, one can now sketch the proof of  Theorem~\ref{theo:mesures} as its arguments are an adaptation of the ones in \cite{FF,FF2}. 

\begin{proof}[Sketch of proof of Theorem~\ref{theo:mesures}]
If $\limsup_{\eps\rightarrow 0}  \| u^\eps\|_{L^\infty(\R,L^2(G))}=0$, then the map given by $\Gamma_t d\gamma_t=0$ for all $t\in \R$ answers our problem.
Hence, by dividing $u^\eps$ by $\limsup_{\eps\rightarrow 0}  \| u^\eps\|_{L^\infty(\R,L^2(G))}$ if necessary, we can assume that 
$$
\limsup_{\eps\rightarrow 0}  \| u^\eps\|_{L^\infty(\R,L^2(G))}=1.
$$ 
We then consider the quantities 
$ \ell_\eps(\theta, \sigma)$ defined in~(\ref{def:leps}) and we observe the three following facts:
\begin{enumerate}
\item 
For any $\theta\in {\mathcal C}_c^\infty(\R)$ and $\sigma\in {\mathcal A}_0$, the family $\ell_\eps(\theta, \sigma)$ is bounded and there exists a subsequence $(\eps_k(\sigma))_{ k\in \mathbb N}$ such that $\ell_{\eps_k(\sigma)}(\theta, \sigma)$ has a limit $\ell(\theta,\sigma)$.
\item 
Using the separability of ${\mathcal C}_c^\infty(\R)\times {\mathcal A}_0$ and a diagonal extraction, one can find a sequence $(\eps_k)_{k\in \mathbb N}$ such that for all $\theta\in {\mathcal C}_c^\infty(\R)$  and  $\sigma\in {\mathcal A}_0$, the sequence $(\ell_{\eps_k}(\theta, \sigma))_{k\in \mathbb N}$ has a limit $\ell(\theta, \sigma)$ and the sequence $(\| u^{\eps_k}\|_{L^\infty(\R,L^2(G))})_{k\in \mathbb N}$ converges to 1.
\item 
The map $(\theta,\sigma )\mapsto \ell(\theta,\sigma)$ constructed at point (2) satisfies \eqref{eq_cor_ellpos}.
It also satisfies \eqref{eq_cor_ellL1}
for symbols of the form $\tau(x,\lambda)=a(x)b(\lambda)$, see~\eqref{simplebound}, therefore for all $\sigma\in {\mathcal A}$ (since the algebraic span of such $\tau$ is dense in~${\mathcal A}$, see~Section~\ref{sec:3.1}).
 \end{enumerate}
We conclude with Proposition~\ref{cor2_prop:states}.
\end{proof}

\subsection{Some comments on time-dependent states}

The result in Proposition \ref{cor2_prop:states}
 calls for  some  comments 
which will not be  used in the following paper but are  of interest in themselves. 
Indeed, with
a straightforward adaptation of the arguments given for the proof of Proposition~\ref{prop:C*A}, we obtain  an analogue description for the closure of 
$C^\infty(J)\otimes {\mathcal A}_0$
if $J$ is a compact interval, 
and of $C^\infty_c(\R)\otimes {\mathcal A}_0$ if $J=\R$.
Note first that in both cases, the closure is for the norm
$$
\|\tau\|_{C_0(J,{\mathcal A})}:=
\inf\left\{\sum_j \|\theta_j\|_{L^\infty(J)}\| \sigma_j\|_{\mathcal A}
\ : \ \tau =\sum_j \theta_j \sigma_j  \right\},
$$
and we identify them respectively with the $C^*$-algebra
 $C(J,{\mathcal A})$ 
of continuous functions on $J$ valued in  ${\mathcal A}$
when $J$ is a compact interval 
and with the Banach space
 $C_0(\R,{\mathcal A})$ 
of continuous functions on~$\R$ valued in~${\mathcal A}$ and vanishing at infinity when $J=\R$.
In order to unify the presentation, 
we may write 
$C_0(J,{\mathcal A})$ for $C(J,{\mathcal A})$ when $J$ is a compact interval of $\R$.

\medskip 

Then, an analysis similar to the one of the proof of Proposition~\ref{prop:C*A}
 gives that the $C^*$-algebra $C_0(J,{\mathcal A})$ for $J$ compact interval and  for $J=\R$ is a separable $C^*$-algebra of type~1.
It is not unital  but admits an approximation of identity. 
Its spectrum may be identified with $J\times G\times \widehat G$.
Its states may be identified with the elements $\Gamma
 \, d\gamma$ in 
${\mathcal M}_{ov}^+(J\times G\times \widehat G)$
satisfying 
$$
\int_{J\times G\times \widehat G} {\rm Tr} \left(\Gamma(t,x, \lambda)\right) d\gamma(t,x, \lambda) =1,
$$
via 
$$
\forall \sigma\in C_0(J,{\mathcal A})\qquad
\ell(\sigma) = \int_{J\times G\times \widehat G} {\rm Tr}\left(\sigma(t,x,\lambda) \Gamma(t,x, \lambda)\right) d\gamma(t,x, \lambda).
$$
Finally, 
a map  $\ell: {\mathcal C}_c^\infty(\R)\times {\mathcal A}_0\to \C $  as in Proposition~\ref{cor2_prop:states}.
extends uniquely into a continuous linear map on $L^1(\R)\times {\mathcal A}$, 
and also into a state of $ C(J,{\mathcal A})$ up to the normalisation $|J|$ for any compact interval.
Using the characterisation above and the uniqueness, 
we can define a pair $(\gamma, \Gamma)\in \widetilde {\mathcal M}_{ov}^+(\R\times G \times \widehat G)$, 
unique up to equivalence,  such that 
$$
\forall \sigma\in {\mathcal A}_0,
\quad
\forall \theta \in {\mathcal C}_c^\infty(\R),
\qquad
\ell(\theta,\sigma) 
= \int_{\R\times G\times \widehat G} \theta(t)\ {\rm Tr}\left(\sigma(x,\lambda) \Gamma(t,x, \lambda)\right) d\gamma(t,x, \lambda);
$$
here $\gamma$ is a measure on $\R\times G\times \widehat G$.
This is a weaker result than the one obtained in Proposition~\ref{cor2_prop:states}, which states that the measure above is absolutely continuous with respect to $dt$, and this explains why  we proceed in this manner.

\medskip

Note that we can proceed as in Proposition \ref{prop:states} and Corollary \ref{cor1_prop:states} to obtain a description of the dual of $C_0(J,\mathcal A)$ that we will use later on:
\begin{corollary}
\label{cor_dualC0}
The topological dual $C_0(J,\mathcal A)^*$ of $C_0(J,\mathcal A)$
may  be identified as a Banach space with 
${\mathcal M}_{ov}(J\times G\times \widehat G)$ in the following way:
\begin{itemize}
\item 
If $\ell:C_0(J,\mathcal A)\to \C$ is a continuous linear form, 
then there exists a unique element $\Gamma d\gamma\in {\mathcal M}_{ov} 	(J\times G\times \widehat G)$ satisfying \eqref{eq_prop_elltrGgno1}. 
\item 
Conversely,  \eqref{eq_prop_elltrGgno1} defines a continuous linear form $\ell$ on $C_0(J,\mathcal A)$.
\end{itemize}
Moreover, we have the following properties:
\begin{enumerate}
\item $\|\ell\|_{C_0(J,\mathcal A)^*} = \|\Gamma d\gamma\|_{{\mathcal M}}$.
\item
$\Gamma d\gamma \geq 0$ if and only if $\ell(\sigma)\geq 0$ for all any positive element $\sigma$ in the $C^*$ algebra $C_0(J,\mathcal A)$. 
\end{enumerate}
\end{corollary}

\subsection{Symbolic properties of the eigenprojectors}
\label{subsec:preliminaries}

In this section, we analyse the  fields of the spectral projectors $\Pi_n^{(\lambda)}$ of $H(\lambda) $. 
We use the notion of homogeneous symbols introduced in~\cite{FF} (see Definition~4.1 therein). 

Following~\cite{FR} (Section 5.2 for any graded nilpotent Lie group and Section 6.5 for the Heisenberg group), the class $S^m$  of  {\it symbols 
of order $m$} in $G$ consists of   fields of operators 
$\sigma(x,\lambda)$ such that 
for each $\alpha,\beta\in \N^n$ and $\gamma\in \R$
we have
 \begin{equation}
\label{eq_def_Smrhodelta}
\sup_{x\in G, \lambda\in \widehat G}
\|({\rm Id} +H(\lambda))^{\frac{ [\alpha]-m +\gamma}2}
X_x^\beta\Delta^\alpha \sigma(x,\lambda) 
({\rm Id} +H(\lambda))^{-\frac{\gamma}2 }\|_{\mathcal L(\mathcal H_\lambda)}<\infty.
\end{equation}
Using the dilation induced on $\widehat G$ by the one of $G$, one defines for $m\in\R$,
 $m$-homogeneous fields of operators $\sigma(x,\lambda)$ by asking that 
$
\sigma(x,\eps\cdot \lambda) = \eps^{m} \sigma(x,\lambda)$
for all $x\in G$, almost all $\lambda\in \widehat G$ and $d\eps$-almost all $\eps>0$ (in the preceding formula, 
 $\eps\cdot \lambda=\eps^2\lambda$ for $\lambda\in {\mathfrak z}^*$ and $\eps\cdot (0,\omega)=(0,\eps\omega)$ for $\omega\in{\mathfrak v}^*$).
In parallel to what is done in the Euclidean setting, one then defines {\it regular 
$m$-homogeneous symbols} as the set $\dot S^m$ of $m$-homogeneous fields of operators  $\sigma(x,\lambda)$ which satisfy
for any $\alpha,\beta\in \N^n$, $\gamma\in \R$:
\begin{equation}
\label{eq_def_dotSmrhodelta}
\sup_{\substack{\lambda\in \widehat G\\ x\in G}}
\| H(\lambda)^{\frac{[\alpha]-m+\gamma}{2}}
X_x^\beta \Delta^\alpha \sigma(x,\lambda) H(\lambda)^{-\frac{\gamma}{2}}\|_{{\mathcal L}(\mathcal H_\lambda)} <\infty.
\end{equation}

In both the inhomogeneous and homogeneous case, it was proved that an equivalent characterisation is 
\eqref{eq_def_Smrhodelta} and \eqref{eq_def_dotSmrhodelta} respectively,
for all $\alpha,\beta$ but only $\gamma=0$.

\begin{proposition}
\label{prop_PiindotS}	
Let $n\in \N$.
 The spectral projectors $\Pi_n^{(\lambda)}$ associated with $H(\lambda)$, $\lambda\in {\mathfrak z}^*\setminus\{0\}$,   form a field $\Pi_n$ of operators which is a homogeneous symbol in $\dot S^0$.

Consequently, for any function $\psi\in C^\infty(\R)$ with $\psi\equiv 0$ on $(-\infty,1/2)$ and $\psi\equiv 1$ on $(1,+\infty)$, 
$\psi( H) \Pi_n$ is  in $S^0$.	
Moreover,  for every $\lambda\in \mathfrak z^*\setminus\{0\}$, 
 $\psi(  u H(\lambda)) \Pi_n $  converges to $  \Pi_{n}$ 
in the strong operator topology (SOT) of $L^2(\mathfrak p_\lambda)$ as $u\to 0$.
Furthermore, $\psi(  u H) \Pi_n$
 converges to $\Pi_n  $ 
in SOT for $L^\infty(\widehat G)$
as $u\to 0$.
\end{proposition}

\begin{remark} A remark on the notations: we shall use $\Pi_n$ when denoting the field of operators acting on $L^2(\mathfrak p_\lambda)$ and write $\Pi_n^{(\lambda)}$ when some $\lambda\in \mathfrak z^*$ is fixed. 
\end{remark}

The proof of Proposition \ref{prop_PiindotS} relies on 
the spectral expression:
\begin{equation}
	\label{eq_Pin_Cn}
\Pi_n^{(\lambda)}= {1\over 2i\pi} \oint _{{\mathcal C}_n} (|\lambda|^{-1}H(\lambda)-z)^{-1} dz, 
\end{equation}
where ${\mathcal C}_n$ is  any circle of the complex plane with centre $2n+d$ and  radius $\rho\in(0,2)$, 
and the following lemma:

\begin{lemma}
\label{lem_lambdaindotS}
The field of operators $|\lambda|I_{L^2(\mathfrak p_\lambda)}$, $\lambda\in {\mathfrak z}^*\setminus\{0\}$, yields a homogeneous regular symbol in $\dot S^{2}$.
\end{lemma}

\begin{proof}[Proof of Lemma \ref{lem_lambdaindotS}]
The first thing to notice is that $|\lambda| H(\lambda)^{-1}$ is a bounded operator (as a self-adjoint operator with bounded eigenvalues). Then, we look at $\Delta_{q}|\lambda|I_{L^2(\mathfrak p_\lambda)}$.
The kernel corresponding to the symbol 
$|\lambda|I_{L^2(\mathfrak p_\lambda)}$ is the distribution $\delta_{v=0}\otimes {\mathcal F}^{-1}_{\mathfrak z} |\lambda|$, 
so the corresponding convolution operator on $G$ acts only on the central component.
Furthermore $\Delta_{q}|\lambda|I_{L^2(\mathfrak p_\lambda)}=0$ for $q=v_j$. 
For $q=z_k$,
$\Delta_{q}|\lambda|I_{L^2(\mathfrak p_\lambda)}$ is up to a constant the group Fourier transform of the distribution 
$\delta_{v=0}\otimes {\mathcal F}^{-1}_{\mathfrak z} \partial_{\lambda_k} |\lambda|$, so
\begin{align*}
\sup_{\lambda\in {\mathfrak z}^*\setminus\{0\}}
\|\Delta_{q}|\lambda|I_{L^2(\mathfrak p_\lambda)}\|_{{\mathcal L}(L^2(\mathfrak p_\lambda)}
&=
\| f\mapsto f* \left(\delta_{v=0}\otimes {\mathcal F}^{-1}_{\mathfrak z} \partial_{\lambda_k} |\lambda|\right)\|_{{\mathcal L}(L^2(G))}
\\&\leq
\| g\mapsto g* \left({\mathcal F}^{-1}_{\mathfrak z} \partial_{\lambda_k} |\lambda|\right)\|_{{\mathcal L}(L^2(\mathfrak z))}
=\sup_{\lambda\in \mathfrak z^*\setminus \{0\}}\partial_{\lambda_k} |\lambda|
<\infty.
\end{align*}
For $q=z_jz_k$,
$\Delta_{q}|\lambda|I_{L^2(\mathfrak p_\lambda)}$ is up to a constant the group Fourier transform of the distribution 
$\delta_{v=0}\otimes {\mathcal F}^{-1}_{\mathfrak z} \partial_{\lambda_j}\partial_{\lambda_k} |\lambda|$.
We see that for $q=z_k$, $[q]=2$ and 
$$
H(\lambda)^{-1+\frac{[q]}2} \Delta_{q}|\lambda| = \Delta_{q}|\lambda|.
$$
Therefore 
$$
\sup_{\lambda\in {\mathfrak z}^*}\|H(\lambda)^{-1+\frac{[q]}2}\Delta_{q}|\lambda|I_{L^2(\mathfrak p_\lambda)}\|_{{\mathcal L}(L^2(\mathfrak p_\lambda))}<\infty.
$$
Recursively, we obtain for any monomial $q$, we have the estimates required by~(\ref{eq_def_dotSmrhodelta}):
$$
\sup_{\lambda\in {\mathfrak z}^*}\|H(\lambda)^{-1+\frac{[q]}2}\Delta_{q}|\lambda|I_{L^2(\mathfrak p_\lambda)}\|_{{\mathcal L}(L^2(\mathfrak p_\lambda))}<\infty,
$$  
where $[q]$ denotes the degree of homogeneity of $q$. 
This shows Lemma \ref{lem_lambdaindotS}.
\end{proof}

\begin{proof}[Proof of Proposition \ref{prop_PiindotS}]
In order to show that the field of operators consisting of the~$\Pi_n^{(\lambda)}$ is in~$\dot S^0$, 
it suffices to show that 
$\left(H(\lambda)-|\lambda|z\right)^{-1}$
 is in $\dot S^2$ 
with uniform semi-norms estimates of the form~(\ref{eq_def_dotSmrhodelta})
with respect to $z\in {\mathcal C}_n$ 
because of~\eqref{eq_Pin_Cn}, Lemma \ref{lem_lambdaindotS}
and 
$$
\left(|\lambda|^{-1}H(\lambda)-z\right)^{-1} 
= 
|\lambda|\left(H(\lambda)-|\lambda|z\right)^{-1}. 
$$
The rest of the statement will then follow by \cite{FF} (see Section~4.2 therein).

\medskip 

The semi-norms estimates of the form~(\ref{eq_def_dotSmrhodelta}) with $\alpha=\beta=0$ and $\gamma=0$ are satisfied for all $z\in {\mathcal C}_n$ with
$$
\sup_{\lambda\in \mathfrak z^*\setminus\{0\}}
\|H(\lambda)\left(H(\lambda)-|\lambda|z\right)^{-1}\|_{{\mathcal L}(L^2(\mathfrak p_\lambda))}
\leq \sup_{n\in \N} \frac{2n+d}{2n+d-\rho}<\infty.
$$
For $q=v_j$ or $z_k$, 
since $\Delta_q I_{L^2(\mathfrak p_\lambda)}=0$,
the Leibniz formula implies
$$
\Delta_q\left(H(\lambda)-|\lambda|z\right)^{-1}
=
-\left(H(\lambda)-|\lambda|z\right)^{-1} 
\Delta_q\left(H(\lambda)-|\lambda|z\right)
\left(H(\lambda)-|\lambda|z\right)^{-1}.
$$
Note that 
$$
\Delta_q\left(H(\lambda)-|\lambda|z\right)
=
\Delta_q H(\lambda)-
z\Delta_q|\lambda|I_{L^2(\mathfrak p_\lambda)},
$$
and that both terms in the right-hand side 
are in $\dot S^{2-[q]}$
(for the first one see \cite{FF}, Example~4.5, for the second  by Lemma \ref{lem_lambdaindotS}), 
so $\Delta_q\left(H(\lambda)-|\lambda|z\right)\in \dot S^{2-[q]}$.
This and the estimates above yield that
\begin{align*}
&\|H(\lambda)^{1+\frac {[q]}2} \Delta_q\left(H(\lambda)-|\lambda|z\right)^{-1} \|_{{\mathcal L}(L^2(\mathfrak p_\lambda))}
\\&\quad\leq 
\|H(\lambda)\left(H(\lambda)-|\lambda|z\right)^{-1} \|_{{\mathcal L}(L^2(\mathfrak p_\lambda))}^2 
\ \| H(\lambda)^
{\frac{[q]}{2}}
\Delta_q\left(H(\lambda)-|\lambda|z\right)H(\lambda)^{-1} \|_{{\mathcal L}(L^2(\mathfrak p_\lambda))},
 \end{align*}
so its supremum over $\lambda\not=0$ is finite.
Proceeding recursively shows
$(H(\lambda)-|\lambda|z)^{-1}\in \dot S^{-2}$ and thus $(|\lambda|^{-1}H(\lambda)-z)^{-1}\in\dot S^0$, 
with uniform semi-norms estimates 
with respect to $z\in {\mathcal C}_n$.
This yields Part (1).
\end{proof}

\subsection{Approximation of $H$-diagonals and anti-$H$-diagonals symbols}
\label{subsec_approxPinproj}
We now use the symbolic properties of Section \ref{subsec:preliminaries} for the projections $\Pi_n^{(\lambda)}$  to decompose symbols as described in Lemma~\ref{lem_sigmad+a}. 

We first notice that 
if $\sigma=\{\sigma(\lambda), \lambda\in \widehat G\}\in S^{-\infty}$  is a smoothing symbol independent of $x$, we can then define the symbol
$\sigma^{(n,n')}:=\Pi_n\sigma \Pi_{n'}$
for each $n,n'$.
For $n=n'$, $\sigma^{(n,n')}$ commutes with $H$; 
for general pairs of integers $(n,n')$, it satisfies
$$	
H(\lambda) \, \sigma^{(n,n')} (\lambda) =|\lambda| (2n+d) \sigma^{(n,n')}(\lambda)
\quad\mbox{and}\quad
\sigma^{(n,n')} H(\lambda) = |\lambda|  (2n'+d) \sigma^{(n,n')}(\lambda).
$$
However, $\sigma^{(n,n')}$ is usually not smoothing but may be weakly approximated by the smoothing symbols $\psi(u H)\Pi_n  \sigma \Pi_{n'}\psi(u H)$ with $u\to 0$ with $\psi$ as in Proposition \ref{prop_PiindotS}.

\medskip

Proposition \ref{prop_PiindotS} allows us to modify the previous construction to symbols which also depend on $x\in G$ in order to obtain  the following weak approximation of the $\Pi_n$-projections of smoothing symbols:

\begin{corollary}
\label{cor_prop_PiindotS}
We fix a smooth function $\psi:\R \to [0,1]$
satisfying $\psi\equiv 0$ on $(-\infty,1/2)$ and $\psi\equiv 1$ on $(1,+\infty)$.
Let $\sigma \in S^{-\infty}$.

\begin{enumerate}
\item For each $n,n'\in \N$ and $u\in (0,1]$, 
	the symbol
\begin{equation}
\label{eq_sigmann'u}	
\sigma^{(n,n',u)} :=
\psi(u H)\Pi_n  \sigma \Pi_{n'}\psi(u H)
\end{equation}
is smoothing, i.e. $\sigma^{(n,n',u)}\in S^{-\infty}$. 
Moreover,  for every $x\in G$ and $\lambda\in \mathfrak z^*\setminus\{0\}$, 
 $\sigma^{(n,n',u)}(x,\lambda)$  converges to $ \sigma^{(n,n')}(x,\lambda):=\Pi_n  \sigma(\lambda) \Pi_{n'}$ 
in SOT of $L^2(\mathfrak p_\lambda)$ as $u\to 0$;
and $\sigma^{(n,n',u)}(x,\cdot)$  converges to $ \sigma^{(n,n')}=\Pi_n  \sigma(x,\cdot) \Pi_{n'}$ 
in SOT for $L^\infty(\widehat G)$
as $u\to 0$, for every $x\in G$.
\item If $\sigma \in \mathcal A_0$, then 
$\sigma^{(n,n',u)}\in \mathcal A_0$ for every $n,n'\in \N$ and $u\in (0,1)$.
\item If $\sigma$ is supported in $K\times (\mathfrak z^*\setminus\{0\})$ with $K$ a compact of $G$ and vanish identically near $\lambda=0$, then 
$\sigma^{(n,n',u)} \in \mathcal A_H$
for every $n,n'\in \N$ and $u\in (0,1)$.

\item If $\sigma\in \mathcal A_H$, 
we can write $\sigma$ as the finite sum 
$\sigma =\sum_{n,n'} \sigma^{(n,n')}$ 
and for $u$ small enough we have 
$\sigma^{(n,n')} = \sigma^{(n,n',u)}$.	

\item If $\sigma\in \mathcal A$ then $\sigma^{(n,n',u)}\in \mathcal A$ for every $n,n'\in \N$ and $u\in (0,1)$. If in addition $\sigma$ is a positive element in the $C^*$-algebra $\mathcal A$, then so is $\sigma^{(n,n',u)}$ for $n=n'$.
\end{enumerate}
\end{corollary}

\begin{proof}
Part (1) follows from 	Proposition \ref{prop_PiindotS} and \cite{FF}.
Parts (2), (3) and (4) follows from Part (1).
Let $\sigma\in \mathcal A$.
Since  $\sigma\mapsto \sigma^{(n,n',u)}$ is linear and 
$\|\sigma^{(n,n',u)}(x,\cdot)\|_{L^\infty(\widehat G)}\leq 
\|\sigma(x,\cdot)\|_{L^\infty(\widehat G)}$, 
 $\sigma^{(n,n',u)}\in \mathcal A$.
Furthermore, for any $\sigma\in \mathcal A$, 
$(\sigma^*\sigma)^{(n,n,u)}$ is positive in $C_0(G)\otimes L^\infty(\widehat G)$, therefore also in $\mathcal A$. 
This shows Part (5). 
\end{proof}

\begin{remark}
\label{rem_pf_lem_sigmad+a}
	Note that Part (4) of Corollary \ref{cor_prop_PiindotS} proves Lemma \ref{lem_sigmad+a}.
\end{remark}

The same ideas also gives the approximations of symbols in $\mathcal A_0$ by symbols in $\mathcal A_H$:

\begin{corollary}
	\label{cor_cv_AH0}	Let $\sigma \in {\mathcal A}_0$.
Then there exists a sequence $(\sigma_n)_{n\in \N}$ of symbols 
 in ${\mathcal A}_H$ converging to $\sigma$ 
in the following sense:
\begin{itemize}
\item 	for each $\lambda\in \mathfrak z^*\setminus\{0\}$, we have the convergence 
$\sigma_n(x,\lambda) \longrightarrow \sigma(x,\lambda)$
in SOT of $L^2(\mathfrak p_\lambda)$ as $n\to +\infty$ uniformly in $x\in G$, and 
\item we have the convergence 
$\sigma_n (x,\cdot) \to \sigma(x,\cdot)$
in SOT of $L^\infty(\widehat G)$ as $n\to +\infty$ uniformly in $x\in G$.
\end{itemize}
\end{corollary}

Corollary \ref{cor_cv_AH0} 
follows from Proposition \ref{prop_PiindotS} and Proposition~4.6 of~\cite{FF} together with a smooth cut-off in $\lambda\in \mathfrak z^*$.
The latter property is granted by the following lemma:

\begin{lemma}
\label{lem_g(lambda)}
Let $\sigma\in {\mathcal A}_0$.
\begin{enumerate}
\item 
For any $g\in {\mathcal S}({\mathfrak z}^*)$, 
 the symbol given by $g(\lambda)\sigma(x,\lambda)$ is in ${\mathcal A}_0$.
\item Let $(g_n)_{n\in \N}\subset C^\infty_c({\mathfrak z}^*\setminus\{0\})$ be a sequence of functions which are bounded by 1 and converges to the constant function 1 pointwise.
Then $(g_n \sigma(x,\cdot))_{n\in \N}$ converges to  $\sigma(x,\cdot)$ in SOT of $L^\infty(\widehat G)$ uniformly in $x\in G$.
\end{enumerate}
\end{lemma}
\begin{proof}
Let $\kappa_\sigma$ be 
the kernels associated with the symbols which may be identified with a map in $C_c^\infty(G;{\mathcal S}(G))$.
Part (1) follows from 
the kernel $\kappa_g$ associated with the symbol
$g(\lambda)$ being the central distribution $\delta_{v=0}\otimes {\mathcal F}_{\mathfrak z}^{-1} g$ and 
  the kernel associated with $g(\lambda)\sigma$ being $\kappa_g * \kappa_\sigma$.
  Part (2) follows from the Lebesgue dominated convergence theorem and the Plancherel formula.
\end{proof}

\begin{remark}
\label{rem_f(H)}
		If $\sigma = \{\sigma(x,\pi), (x,\pi)\in G\times \widehat G\}\in S^{-\infty}$, 
then constructing as above $\psi(u H)\Pi_n  \sigma \Pi_{n}\psi(u H)$ yields a smoothing symbol commuting with $H(\lambda)$.
The closure of their span form a much larger class than the class of spectral multipliers  of $H(\lambda)$. Indeed, spectral multipliers are constant on the vector sets $\mathcal V_n$ which is not the case of  $\psi(u H)\Pi_n  \sigma \Pi_{n}\psi(u H)$. The reader can refer to~\cite[Chapter 4]{FR} for considerations on symbols that are functions of $H(\lambda)$.
\end{remark}

\subsection{The sub-$C^*$-algebra $\mathcal B$ of $H$-commuting symbols}\label{sec:Hdiag}

In order to prove Theorem~\ref{theo:schro}, we shall use symbols in $\mathcal A$ which commute with $H(\lambda)$. The set ${\mathcal B}$ of such symbols satisfies the following properties: 

\begin{lemma}
\label{lem_stateB}
Let $\mathcal B$ be the sub-$C^*$-algebra of $\mathcal A$ consisting of symbols $\sigma\in\mathcal A$ which commute with $H(\lambda)$, 
i.e. $\sigma(x,\lambda) H(\lambda) = H(\lambda) \sigma(x,\lambda)$ for almost every $(x,\lambda)\in G\times \widehat G$.
\begin{enumerate}
\item 	The space $\mathcal B$ contains 
	the symbols of the form $a(x) \sigma(\lambda)$ with $a\in C_c^\infty(G)$ and 
	$\sigma$ smoothing and commuting with $H(\lambda)$, and 
	the algebraic span of these symbols are dense in $\mathcal B$.
\item 	
The states of $\mathcal B$ are in one-to-one correspondence as in Proposition \ref{prop:states} with the measures $\Gamma d\gamma\in \mathcal M_{ov}^+(G\times \widehat G)$ such that $\Gamma = \sum_{n\in \N} \Pi_n \Gamma \Pi_n$.
\end{enumerate}
	\end{lemma}

\begin{proof}[Proof of Lemma \ref{lem_stateB}]
Part (1) is readily checked. Let us prove Part (2).
Let $\ell$ be a state of $\mathcal B$.
For any $N\in \N$ and $u\in (0,1)$, 
Part (5) of Corollary \ref{cor_prop_PiindotS}  implies that 
$$
\ell_{N,u} (\sigma) := \sum_{n=0}^N \ell (\sigma^{(n,n,u)}),
$$
 defines a continuous linear functional  $\ell_{N,u}$ on $\mathcal A$ which is a state or 0.
 We denote by $\Gamma_{N,u}d\gamma_{N,u}\in \mathcal M_{ov}^+(G\times \widehat G)$ the measure corresponding to $\ell_{N,u}$ by Proposition \ref{prop:states}.
We observe that for any $N_1\leq N_2$ and $2u_1 < u_2 < u_3/2$ 
$$
\ell_{u_1,N_1} (\sigma)
=
\ell_{u_2,N_2} ( \sum_{n=0}^{N_1} \sigma^{(n,n,u_3)}).
$$
The uniqueness in Proposition \ref{prop:states}
implies that there exists $\Gamma d\gamma\in \mathcal M_{ov}^+(G\times \widehat G)$ of mass 1, commuting with $H(\lambda)$ and such that $\Gamma_{N,u} d\gamma_{N,u} = \sum_{n=0}^N \psi(u H)\Pi_n  \Gamma \Pi_{n'}\psi(u H)   d\gamma$. 
This implies Part~(2).
\end{proof}

It readily follows from Lemma~\ref{lem_stateB} that the pure states of the $C^*$-algebra $\mathcal B$ are given either by  $|v \rangle \langle v |\,\delta_{x_0}\otimes \delta_{\pi^{\lambda_0}}$ 
for some $v\in \mathcal V_n$ when $\lambda_0\in \mathfrak z^*\setminus\{0\}$ and $x_0\in G$,
or by $\delta_{x_0}\otimes \delta_{\pi^{(0,\omega_0)}}$
for some $\omega_0\in \mathfrak v^*$ and $ x_0\in G$. One can thus describe easily the dual of $\mathcal B$ by identification  with the  subset 
$$
S:=\{(x,\lambda_0, n) \in 
 G\times  \widehat G \times \N \ : \
 n=0 \ \mbox{if} \ \lambda_0\not\in \mathfrak z^*\setminus\{0\}
 \} \ \mbox{of} \ G\times \widehat G\times \N.
 $$

\section{Proof of Theorems~\ref{theo:Egorov} and~\ref{theo:schro}}\label{sec:proofs}

The core of our results are Theorems~\ref{theo:Egorov} and~\ref{theo:schro} from which Theorem~\ref{theorem1} will derive in Section \ref{sec:prooftheorem1}, 
so we focus on these two first statements. 
Their proofs rely on a careful analysis of the commutator $[{\rm Op}_\eps(\sigma),\eps^2 \Delta_G] $ based on Equation~\eqref{prop:commutatorDelta}.
Considering anti-$H$-commuting symbols in Section~\ref{sec:anticom} allows us to prove first Point (i) of both Theorems~\ref{theo:Egorov} and~\ref{theo:schro}. 
We successively prove Points (ii)  of the two theorems in Sections~\ref{proof(ii)} and (iii) of Theorem~\ref{theo:schro} in Section~\ref{proof(iii)} respectively.

\subsection{Spectral decomposition of the time-averaged semi-classical measure}\label{sec:anticom}
Here, we  prove Part (i) of Theorems~\ref{theo:schro} and~\ref{theo:Egorov}, together with equation~(\ref{eq:decomp}).
We take $\sigma\in{\mathcal A}_0$. 
The Schr\"odinger equation \eqref{eq:schro} yields
$$
i\eps^\tau {d\over dt} \left({\rm Op}_\eps(\sigma) \psi^\eps(t) ,\psi^\eps(t) \right)_{L^2(G)} 
=  \left(\left[{\rm Op}_\eps(\sigma) , -\frac{\eps^2}2\Delta_G \right] \psi^\eps(t) ,\psi^\eps(t) \right)_{L^2(G)}.
$$
By use of~\eqref{prop:commutatorDelta}, we obtain 
\begin{eqnarray*}
{1\over 2}  \left({\rm Op}_\eps\left([\sigma,H(\lambda) \right] \psi^\eps(t) ,\psi^\eps(t) \right)_{L^2(G)}
&=&  
\frac 12\left(\left[{\rm Op}_\eps(\sigma) , -\eps^2\Delta_G \right] \psi^\eps(t) ,\psi^\eps(t) \right)_{L^2(G)} + O(\eps) +O(\eps)\\
 &=&\frac i2 \eps^\tau {d\over dt} \left({\rm Op}_\eps(\sigma) \psi^\eps(t) ,\psi^\eps(t) \right)_{L^2(G)}+O(\eps).
 \end{eqnarray*}
Therefore, 
\begin{equation}
\label{eq:commutator}
\int_{\R} \theta(t)   \left({\rm Op}_\eps\left([\sigma,H(\lambda) ]\right) \psi^\eps(t) ,\psi^\eps(t) \right)_{L^2(G)} dt = O(\eps^\tau)+O(\eps).
\end{equation}

\subsubsection{Proof of Theorem~\ref{theo:schro} (i)}
Firstly, 
taking the limit as $\eps\to0$ in \eqref{eq:commutator}  for a subsequence defining a semi-classical measure $\Gamma_td\gamma_t$, it gives 
 for all $\sigma\in {\mathcal A}_0$ and for all $\theta\in{\mathcal C}_c^\infty (\R)$, 
$$
\int_{\R}\theta(t)\int_{G\times \widehat G} {\rm Tr}  \left( [\sigma(x,\lambda ),H(\lambda)]\, \Gamma_t(x,\lambda)\right)d\gamma_t(x,\lambda)dt=0.
$$
We apply this to any symbol $\sigma^{(n,n',u)}\in \mathcal A_0$ given by \eqref{eq_sigmann'u}, see Corollary \ref{cor_prop_PiindotS}.
By Corollary \ref{cor_dualC0},
this implies that 
the element in $\mathcal M_{ov}(\R\times G\times \widehat G)$ given by  
$$
[H(\lambda) \, , \, 
\Pi_{n'}\psi(u H(\lambda))
\Gamma_t(x,\lambda)
\psi(u H(\lambda))\Pi_n] \,
 d\gamma_t dt 
$$ 
is zero. 
For any pair of integers with $n\not=n'$, 
taking $u\to 0$ implies that
$\Pi_{n'} \Gamma_t \Pi_{n}=0$
for almost every $(x,\lambda)\in G\times \widehat G$, this shows Point~(i) of Theorem~\ref{theo:schro}.
As $\Gamma_t$ is a positive compact operator, it  admits the spectral decomposition~(\ref{eq:decomp}).

\subsubsection{Proof of Theorem~\ref{theo:Egorov}  (i)}
Secondly,  we can apply \eqref{eq:commutator} for the symbols 
$$
\tilde \sigma^{(g, n,n',u)}(x,\lambda):= 
g(\lambda) (2|\lambda| (n-n'))^{-1} \sigma^{(n,n',u)}
$$
where $\sigma^{(n,n',u)}$ is as above and $g\in \mathcal S(\mathfrak z^*)$ supported away from 0;
by Lemma \ref{lem_g(lambda)}, 
this  symbol is indeed in $\mathcal A_0$.
We obtain that for $n\not=n'$ 
$$
\int_{\R} \theta(t)   \left({\rm Op}_\eps\left(g \sigma^{(n,n',u)}\right) \psi^\eps(t) ,\psi^\eps(t) \right)_{L^2(G)} dt = O(\eps^{\min (1,\tau)}).
$$
This implies Point~(i) of Theorem~\ref{theo:Egorov} 
for  $\sigma\in \mathcal A_H$ by taking $g\in C_c^\infty(\mathfrak z^*\setminus\{0\})$ identically 1 on the $\lambda$-support of $\sigma$, and $u$ small enough so that 
$\sigma^{(n,n')} = \sigma^{(n,n',u)}$ where 
$\sigma =\sum_{n,n'} \sigma^{(n,n')}$ with the notation of Corollary \ref{cor_prop_PiindotS} Part (2).

\subsubsection{A more precise computation}
In order to prove the rest of Theorems~\ref{theo:Egorov}  and~\ref{theo:schro}, we write down more precisely  the equalities we have obtained above using~\eqref{prop:commutatorDelta} 
\begin{align}
\label{eq_epstauOp}
&i\eps^\tau {d\over dt} \left({\rm Op}_\eps(\sigma) \psi^\eps(t) ,\psi^\eps(t) \right)_{L^2(G)} 
=
\frac 12 \left({\rm Op}_\eps\left[ \sigma, H(\lambda) \right] \psi^\eps(t) ,\psi^\eps(t) \right)_{L^2(G)}
\\&\qquad\qquad-\eps 
\left({\rm Op}_\eps\left(V\cdot \pi^\lambda(V) \sigma\right)
 \psi^\eps(t) ,\psi^\eps(t) \right)_{L^2(G)}
 -\frac {\eps^2}2
\left({\rm Op}_\eps \left(\Delta_G \sigma\right) \psi^\eps(t) ,\psi^\eps(t) \right)_{L^2(G)}.	
\nonumber
\end{align}
It is actually  convenient to use the notation 
$\ell_\eps(\theta,\sigma)$  introduced in \eqref{def:leps}.
By Theorem~\ref{theo:mesures}, we know that 
up to extraction of a subsequence $\eps_k$, 
$
\ell_\eps(\theta,\sigma)
$ has a limit  that we denote by 
$\ell_\infty(\theta,\sigma)$.
With these notations, equation~(\ref{eq_epstauOp}) writes 
 \begin{equation}
 \label{eq_epstauIeps} 
 	-i\eps^\tau \ell_\eps (\theta',\sigma) 
 	=\frac 12 \ell_\eps \left(\theta,\left[\sigma, H(\lambda)\right]\right) 
-\eps  \ell_\eps \left(\theta,V\cdot \pi^\lambda(V)\sigma\right) 
-\frac{\eps^2}2 \ell_\eps \left(\theta,\Delta_G \sigma\right) . 
 \end{equation}

\subsection{Proof of (ii) Theorems~\ref{theo:schro} and~\ref{theo:Egorov}}\label{proof(ii)}

In this paragraph, we prove Parts (ii) in Theorems~\ref{theo:schro} and~\ref{theo:Egorov}.
We consider symbols supported away from $\lambda=0$
and commuting with $H(\lambda)$. We are going to use some properties that
 are summarised in the following technical lemma, the proof of which is in Appendix~B. 

\begin{lemma}
\label{lem_sigma_comH} 
If $\sigma_0\in {\mathcal A}_0$ commutes with $H$
and if  $g\in {\mathcal S}({\mathfrak z}^*)$ is supported away from 0, 
then the symbol $\sigma\in {\mathcal A}_0$  given via 
$\sigma(x,\lambda) = g(\lambda) \sigma_0(x,\lambda)$ 
satisfies the following properties:
\begin{enumerate}
\item
	The symbol $\sigma_1$ given via 
$$
\sigma_1(x,\lambda) = 
\frac {-1}{2i|\lambda|}
\sum_{j=1}^d \left(P_j \pi^\lambda (Q_j) - Q_j \pi^\lambda(P_j)\right)\sigma(x,\lambda) 
$$
is in ${\mathcal A_0}$.
\item
For any  $(x,\lambda)\in G\times \widehat G$,
$$
\left[\sigma_1(x,\lambda),H(\lambda)\right]
= V\cdot \pi^\lambda(V) \sigma(x,\lambda).
$$ 
\item 
For any $n\in \N$, $(x,\lambda)\in G\times \widehat G$,
$$
\Pi_n (V\cdot \pi^\lambda(V) \sigma_1(x,\lambda)) \Pi_n
=\frac 14 
\left((2n+d)	i|\lambda|^{-1} {\mathcal Z}^{(\lambda)} 
-  \Delta_G\right)\Pi_n \sigma(x,\lambda)	\Pi_n .
$$
\end{enumerate}
\end{lemma}

\subsubsection{The case $\tau\in (0,2)$ - Proof of  Part (ii) (1) of Theorems~\ref{theo:schro} and~\ref{theo:Egorov}}
\label{subsubsec_tau02}
Continuing with the setting of Lemma \ref{lem_sigma_comH},
its parts (1) and (2) together with  equation~\eqref{eq_epstauIeps}  applied to the symbol $\sigma_1$ 
yield the development of 
the term:
\begin{align*}
 \ell_\eps \left(\theta,V\cdot \pi^\lambda(V)\sigma\right)
 &=
 \ell_\eps \left(\theta,[\sigma_1,H(\lambda)]\right)
 \\
 \label{5.7}
 &=
 -2i\eps^\tau \ell_\eps (\theta', \sigma_1)
 +2\eps \ell_\eps (\theta, V\cdot \pi^\lambda(V) \sigma_1) +\eps^2 \ell_\eps (\theta, \Delta_G \sigma_1).
\end{align*}
 Plugging this into  \eqref{eq_epstauIeps} 
 shows that we have 
 \begin{eqnarray}\nonumber
	\ell_\eps (\theta',\sigma)
	&=&
	i\eps^{-\tau}\left(-\eps  \ell_\eps \left(\theta,V\cdot \pi^\lambda(V)\sigma\right) 
-\frac{\eps^2}2 \ell_\eps \left(\theta,\Delta_G \sigma\right)\right)\\
\label{keyline}
&= & 2\eps \ell_\eps (\theta', \sigma_1)
 -2i\eps^{2-\tau} \ell_\eps (\theta, V\cdot \pi^\lambda(V) \sigma_1) -\frac {i}2 {\eps^{2-\tau}} \ell_\eps \left(\theta,\Delta_G \sigma\right)
  -i\eps^{3-\tau} \ell_\eps (\theta, \Delta_G \sigma_1)
  \\
\nonumber
& =& O(\eps)+  O(\eps^{2-\tau}).
\end{eqnarray}
Proceeding as in Section \ref{sec:anticom},
this implies the case $\tau\in (0,2)$, i.e. (ii) Part (1) of Theorems~\ref{theo:schro} and~\ref{theo:Egorov}.

\subsubsection{A more precise computation}
\label{subsubsec_moreprecisecomp}
Before focusing on the cases $\tau\geq 2$, 
let us make more explicit the computation in Section \ref{subsubsec_tau02}  (which is valid for any $\tau$) by setting for any $\sigma$ as in Lemma \ref{lem_sigma_comH}:
\begin{equation}
	j_\eps(\theta,\sigma)
	:=
2\ell_\eps (\theta, V\cdot \pi^\lambda(V) \sigma_1)
+\frac 12 \ell_\eps (\theta, \Delta_G \sigma). \label{def:Jeps}
\end{equation}
Rewriting (\ref{keyline}), we obtain 
\begin{equation}\label{def:Jeps'}
j_\eps(\theta,\sigma)=i\eps^{\tau-2}\ell_\eps (\theta',\sigma)
 -2i \eps^{\tau-1} \ell_\eps (\theta', \sigma_1)
+\eps \ell_\eps (\theta, \Delta_G \sigma_1).
\end{equation}
Theorem~\ref{theo:Egorov} (i) and  Lemma~\ref{lem_sigma_comH} (3) together with \eqref{def:Jeps} give for any $\sigma \in \mathcal A_H$, 
\begin{equation}\label{def:Jepsbis}
j_\eps(\theta,\sigma)= {i} 
\sum_{n\in\N} \ell_\eps\left(\theta,\frac{2n+d}{2|\lambda|} {\mathcal Z}^{(\lambda)}\Pi_n\sigma\Pi_n\right)+O(\eps^{\min(1,\tau)}),
\end{equation}
as the terms in $\Delta_G$ cancel each other.

\medskip 

We  set for any $\sigma$ as in Lemma \ref{lem_sigma_comH}:
$$
j_\infty(\theta,\sigma):=
\lim_{\eps_k\to 0} j_{\eps_k}(\theta,\sigma).
$$
Theorem \ref{theo:mesures}
and \eqref{def:Jeps'} imply
\begin{equation}
\label{def:Jepstau>2}
j_\infty(\theta,\sigma) = 0
\quad \mbox{when}\quad \tau>2,
\end{equation}
while 
passing to the limit by Theorem \ref{theo:mesures}, 
 we have for any $\tau >0$:
\begin{equation}\label{Jinfini=}
j_\infty(\theta,\sigma)
=\int_{\R} \theta(t) \int_{G\times \widehat G}
 {i } \sum_{n=0}^\infty \frac{2n+d}{2|\lambda|}\,
	{\rm Tr}  
	 \left(\left( {\mathcal Z}^{(\lambda)} 
\Pi_n \sigma(x,\lambda)\Pi_n \right) \Gamma_{n,t}(x,\lambda)\right)d\gamma_t(x,\lambda) dt.
\end{equation}

\subsubsection{Proof of Parts (2) and (3) of  Theorem \ref{theo:Egorov} (ii)}
We now consider $\sigma\in \mathcal A_H^{(d)}$.
We may assume that  $\sigma=\Pi_n \sigma \Pi_n$ for some fixed $n\in \N$.

\medskip 

In the case $\tau=2$, 
the computations in Section \ref{subsubsec_moreprecisecomp}  give (using first equation~(\ref{def:Jeps'}) and then~(\ref{def:Jepsbis}))
\begin{align*}
&\frac{d}{ds} \ell^{\eps} \left( \theta(\cdot+s), \sigma\circ
\Psi^{-\frac {2n+d}{2|\lambda|} s}
 \right)
 \\&\quad 
 = 
\ell^{\eps} \left( \theta'(\cdot+s), \sigma\circ \Psi^{-\frac {2n+d}{2|\lambda|} s} \right)
+
\ell^{\eps} \left( \theta(\cdot+s), -\frac {2n+d} {2|\lambda|}\mathcal Z^{(\lambda)} \sigma\circ \Psi^{-\frac {2n+d}{2|\lambda|} s}\right)	
\\
&\quad 
= - i j_\eps (\theta(\cdot+s), \sigma\circ \Psi^{-\frac {2n+d}{2|\lambda|} s}) 
\ + \
\ell_{\eps} \left( \theta(\cdot+s), -\frac {2n+d}{2|\lambda|}\mathcal Z^{(\lambda)} \sigma\circ \Psi^{-\frac {2n+d}{2|\lambda|} s}\right)
\ +\ O(\eps)
\\
&\quad 
= O(\eps).
\end{align*}
An integration over $s$  shows Part (1).

\medskip 

Part (3) for $\tau>2$ is proved in a similar manner: using again~(\ref{def:Jeps'}) and~(\ref{def:Jepsbis})
\begin{align}\label{5.12}
\frac{d}{ds}\ell_{\eps} \left( \theta, \sigma\circ \Psi^{-\frac {2n+d}{2|\lambda|} s} \right)
&=
\ell_{\eps} \left( \theta, -\frac {2n+d}{2|\lambda|}\mathcal Z^{(\lambda)} \sigma\circ \Psi^{-\frac {2n+d}{2|\lambda|} s} \right)
\\&= 
{{i}} j_\eps  \left( \theta,  \sigma\circ \Psi^{-\frac {2n+d}{2|\lambda|} s} \right)+O(\eps )
=O(\eps^{{\rm min}(1,\tau-2)}).
\nonumber
\end{align}
This concludes the proof of Theorem \ref{theo:Egorov}.

\subsubsection{Proof  of Part (2) of Theorem \ref{theo:schro}  (ii)} 
Here $\tau=2$. 
For any $\sigma\in \mathcal A_H$ satisfying   $\sigma=\Pi_n \sigma \Pi_n$ for some fixed $n\in \N$, 
the considerations  in Section \ref{subsubsec_moreprecisecomp} (see equations~(\ref{keyline}), (\ref{def:Jeps}) and~(\ref{Jinfini=}))
 give
$$
\ell_\infty(\theta',\sigma) = 
\ell_\infty\left( \theta, {{\frac {2n+d}{ 2 |\lambda|}}} \mathcal Z^{(\lambda)} \sigma\right) \ \mbox{if} \ \tau =2.
$$
The weak density of $\mathcal A_H$ in $\mathcal A_0$ (see Corollary \ref{cor_prop_PiindotS})
implies that 
 Part (2) of Theorem \ref{theo:schro} (ii).
 
\subsubsection{Proof of  Part (3) Theorem \ref{theo:schro} (ii)}
Here $\tau>2$.
For any $\sigma\in \mathcal A_H$ satisfying   $\sigma=\Pi_n \sigma \Pi_n$ for some fixed $n\in\N$,
passing to the limit in~~(\ref{5.12}) give (in view of~(\ref{def:Jepstau>2}))
$$
\ell_\infty( \theta, |\lambda|^{-1} \mathcal Z^{(\lambda)} \sigma) = 0.
$$
The weak density of $\mathcal A_H$ in $\mathcal A_0$ (see Corollary \ref{cor_prop_PiindotS})
implies that  $\Gamma_{n,t}(x,\lambda)d\gamma_t(x,\lambda)$ is invariant under ${\mathcal Z}^{(\lambda)}$ away from $\lambda=0$.
As this is true for every $n\in \N$, 
 the measure ${\bf 1}_{\lambda\not=0}{\rm Tr} (\Gamma_t) d\gamma_t = {\bf 1}_{\lambda\not=0} \sum_{n\in \N}{\rm Tr} (\Gamma_{n,t}) d\gamma_t $
is invariant under the flow $\Psi^s$ defined in \eqref{eq_Psis}.
This concludes the proof of Theorem~\ref{theo:schro} (ii) in view of the discussion at  the end of Section~\ref{sec:enoncetheo}.

\subsection{Proof of Theorem~\ref{theo:schro} (iii)}\label{proof(iii)}
 
 We recall that the measure $\Gamma_t d\gamma_t$ is scalar above $\{\lambda=0\}$, and we can define the measure $d\varsigma_t(x,\omega):= \Gamma_t(x,(0,\omega))d\gamma_t(x,(0,\omega)){\bf 1}_{\lambda=0}$. In order to compute this measure, we observe the following fact. 
 
 \begin{lemma}
 \label{lem_sigma_rest}
 The map $\sigma\mapsto \sigma|_{G\times \{\lambda=0\}}$ is a $C^*$ algebra morphism from $\mathcal A$ to $C_0(G\times \mathfrak v^*)$.
 This map as well as its restriction to the 
 sub-$C^*$-algebra $\mathcal B$ of $H$-commuting symbols
 (see Section \ref{sec:Hdiag})
 are surjective.
  \end{lemma}

We observe that the map $\Theta: \sigma\mapsto \sigma|_{G\times \{\lambda=0\}}$ maps ${\mathcal A}_0$ on  $\mathcal C_c^\infty (G;\mathcal S(\mathfrak v^*))$
in the following way $$
\Theta(\sigma)(x,\omega):=
\sigma(x,(0,\omega)) = \int_G  \kappa_x(v,z) e^{- i \omega\cdot v} dv dz,
$$	
  where we write $\sigma\in \mathcal A_0$  
as $\sigma(x,\lambda) = \widehat \kappa_x (\lambda)$ 
with  the map $x\mapsto \kappa_x $ in $\mathcal C_c^\infty (G;\mathcal S(G))$.
In fact, $\Theta(\mathcal A_0) = C_c^\infty (\mathcal S(\mathfrak v^*))$ and 
we can easily construct a right inverse via
$$
(x\mapsto \phi_x(\omega)) \  \longmapsto \ \mathcal F_G (\mathcal F^{-1}_{\mathfrak v^*} \phi_x  \ \mathcal F^{-1}_{\mathfrak z^*} \chi (z)),
$$
where $\chi\in \mathcal C_c^\infty(\mathfrak z^*)$ with $\chi(0) =1$.
From this, we easily check that $\Theta$ extends to  a $C^*$-algebra morphism from~$\mathcal A$ onto $\mathcal C_0(G\times \mathfrak v^*)$.  
However, we now give below another argument for the surjectivity of $\Theta$ which has the advantage that it also  holds for its restriction to $\mathcal B$.

\begin{proof}[Proof of Lemma \ref{lem_sigma_rest}]
Using the notation just above, 
 $\Theta:\sigma\mapsto \sigma|_{G\times \{\lambda=0\}}$ maps ${\mathcal A}_0$ on  $\mathcal C_c^\infty (G;\mathcal S(\mathfrak v^*))$ and extends to 
 a $C^*$-algebra morphism from~$\mathcal A$ to $\mathcal C_0(G\times \mathfrak v^*)$.  
 The set $\Theta({\mathcal A})$ is a sub-${\mathcal C}^*$-algebra of $\mathcal C_0(G\times \mathfrak v^*)$, their spectrum are included accordingly with equality if and only if they are equal. 
Any state of $\mathcal C_0(G\times \mathfrak v^*) $
is given by a measure in $\mathcal M^+(G\times \mathfrak v^*)$
which may be viewed as an operator valued measure in $\mathcal M^+_{ov}(G\times \widehat G)$ vanishing on $G\times \{\lambda\not=0\}$.
This shows that the spectrum of the commutative algebra $\Theta(\mathcal A)$ is $G\times \mathfrak v^*$,
so $\Theta(\mathcal A) = C_0(G\times \mathfrak v^*)$.
The same argument holds for $\mathcal B$.
\end{proof}

For any $\theta\in \mathcal C_c^\infty (\R)$ and any $\sigma\in \mathcal A_0\cap \mathcal B$,
we decompose  the integrals in $\ell_\infty(\theta', \sigma)$
and $\ell_\infty (\theta, V\cdot \pi^\lambda(V)\sigma)$ over $G\times \widehat G$ as
the sum of two integrals, one over $G\times \{\lambda\not=0\}$ and  one over $G\times \{\lambda =0\}\sim G\times {\mathfrak v}^*$.
The two integrals over $G\times \{\lambda\not=0\}$ are zero by Part (ii)  and by Lemma~\ref{lem_sigma_comH} respectively.
Then, using this when passing to the limit in~\eqref{eq_epstauIeps}, we have the following alternatives. 
\begin{enumerate}
\item If $\tau\in(0,1)$,  the relation
$\ell_\infty(\theta', \sigma) =0$ gives 
$$
\forall \sigma \in \mathcal A_0\cap \mathcal B \qquad
\int_{\R} \theta'(t)
\int_{G\times {\mathfrak v}^*}  \sigma(x,(0,\omega)) d\varsigma_t(x,\omega) dt =0,
$$
whence $\partial_t \varsigma_t=0$ in the sense of distributions by Lemma \ref{lem_sigma_rest}.
\item If $\tau=1$,  using~(\ref{Fourieromega}), 
we obtain
$\ell_\infty(\theta', \sigma) = \ell_\infty (\theta, V\cdot \pi^\lambda(V)\sigma)$, whence 
$$\int_{\R} \theta'(t)
\int_{G\times {\mathfrak v}^*}  \sigma(x,(0,\omega)) d\varsigma_t(x,\omega) dt=
\int_{\R} \theta(t)
\int_{G\times {\mathfrak v}^*}  
\omega\cdot V\sigma(x,(0,\omega)) d\varsigma_t(x,\omega) dt,$$
from which we deduce in the sense of distributions by Lemma \ref{lem_sigma_rest}
$$\partial_t\varsigma_t=   \omega\cdot V\varsigma_t.$$
\item If $\tau>1$, using again~(\ref{Fourieromega}) and Lemma \ref{lem_sigma_rest} yields $\ell_\infty (\theta, V\cdot \pi^\lambda(V)\sigma)=0$
and this implies that the measure $\varsigma_t(x,\omega)$ is invariant under the flow $\Xi^s$, $s\in \R$,
defined by 
$$
\Xi^s:
\left\{\begin{array}{rcl}
G\times \mathfrak v^*
 &\longrightarrow & 
 G\times \mathfrak v^*\\
(x,\omega)
&\longmapsto &
\left({\rm Exp} (s\,\omega\cdot V) x,\omega\right) = \left({\rm Exp} \left(s\,\sum_{j=1}^d \omega_j V_j\right) x,\omega\right)
\end{array}
\right. .$$
And this invariance translates in terms of localisation of the support of  $\varsigma_t$  in view of the discussion of the end of Section~\ref{sec:enoncetheo}, whence Part (iii) (3) of Theorem~\ref{theo:schro} .
\end{enumerate}
In the situation when $\tau\in(0,1]$,  Part (iii) (1) and (2) of Theorem~\ref{theo:schro}  comes from the resolution of the transport equations satisfied by $\varsigma_t$  by using the continuity of $t\mapsto \varsigma_t$ that is proved in the next section.
This concludes the proof of Theorem~\ref{theo:schro}.

\subsection{An improvement in the case $\tau\in (0,1]$}
\label{subsec_tauin01}

Here, we show the following improvement for the case $\tau\in (0,1]$:

\begin{proposition}
\label{prop_tau01} Assume $\tau\in(0,1]$ and consider, as in Theorem \ref{theo:schro}, 
 a semi-classical measure $t\mapsto \Gamma_t d\gamma_t\in L^\infty(\R, {\mathcal M}_{ov}^+(G\times \widehat G))$ corresponding to the family of solutions of \eqref{eq:schrosc} for  an initial data   $(\psi_0^\eps)_{\eps>0}$ which is a
 bounded family in $L^2(G)$.
Then, for any $\sigma\in \mathcal B$, the map $t\mapsto \int_{G\times \widehat G} {\rm Tr}(\sigma \Gamma_t) d\gamma_t$ is locally Lipschitz on $\R$.
\end{proposition}

The proof of Proposition \ref{prop_tau01} relies on \eqref{eq_epstauOp} and  Lemma~\ref{lem_stateB}.

\begin{remark}
\begin{enumerate}
\item  The weak continuity of the map $t\mapsto \Gamma_t d\gamma_t$ granted in Proposition \ref{prop_tau01} allows us  to solve the transport equations of~(ii) Point~(1), and so for~(i) and~(ii) Point~(2).  
\item The proof of Proposition~\ref{prop_tau01} implies that,
under the assumptions of Proposition~\ref{prop_tau01},
if $\eps_k$ is the sub-sequence realising a semi-classical measure $\Gamma_t d\gamma_t$, then we have for all $t\in\R$
  \begin{equation}\label{limittbyt}
 \forall \sigma\in \mathcal A^{(d)}_H \qquad
 \left({\rm Op}_\eps(\sigma)\psi^{\eps_k} (t) ,\psi^{\eps_k}(t) \right) \Tend{k}{+\infty} \int_{G\times \widehat G} {\rm Tr} (\sigma(x,\lambda) \Gamma_t(x,\lambda) )d\gamma_t (x,\lambda),
 \end{equation} 
 meaning that one can pass to the limit $t$-by-$t$ and not only when averaged in time as in the original statement of Theorem~\ref{theo:schro}.  
\end{enumerate}
 \end{remark}

\begin{proof}[Proof of Proposition \ref{prop_tau01}]
For any $\sigma \in \mathcal B$,
 \eqref{eq_epstauOp} gives 
\begin{equation}
\label{eq_ddtell}
{d\over dt} \ell_{\eps,t}(\sigma) = O_{\sigma}(\eps^{1-\tau}),
\quad\mbox{where}\quad
\ell_{\eps,t}(\sigma):=\left({\rm Op}_\eps\left(\sigma\right) \psi^\eps(t) ,\psi^\eps(t) \right),
\end{equation}
in the sense that 
we may assume the distribution $t\mapsto \ell_{\eps,t}(\sigma)$ to be continuous and even $C^1$ on $\R$ 
and that  ${d\over dt} \ell_{\eps,t}(\sigma)$  is uniformly bounded with respect to $t$ in a bounded interval of $\R$ and $\eps\in (0,1)$.
Consider a sequence $(\eps_j)_{j\in \N}$  in $(0,1)$ converging to 0 as $j\to \infty$.
By the Arz\'ela-Ascoli theorem and  \eqref{eq_ddtell},
we can extract  a subsequence $(\eps_{j_k})_{k\in \N}$
such that, as $k\to \infty$,
 $\eps_{j_k}\to 0$ 
and $(\ell_{\eps_{j_k},\cdot}(\sigma))_{k\in \N}$
converges to a continuous function 
$t\mapsto\ell_t(\sigma)$ locally uniformly on $\R$ for all $\sigma\in{\mathcal B}$ (this requires to consider a dense subset of ${\mathcal B}$ and a diagonal extraction procedure).
We proceed as in the proof of Theorem~\ref{theo:mesures} 
 (see also \cite{FF,FF2}) using the $C^*$-algebra $\mathcal B$ with its properties in Lemma \ref{lem_stateB} instead of $\mathcal A$: 
 either $L:=\limsup_{k\to 0}\| \psi^{\eps_{j_k}}_0\|_{L^2(G)}=0$ and $\ell_t=0$
or 
$L^{-1}\ell_t$ is a state of $\mathcal B$ 
for each $t\in \R$. Let
 $\Gamma_t d\gamma_t\in{\mathcal M}_{ov}(G\times \widehat G)$ with $\Gamma_t = \sum_{n\in \N} \Pi_n \Gamma_t \Pi_n$ corresponding to $\ell_t$.
Up to a further extraction of a converging subsequence (for which we keep the same notation), 
we may assume that
 $\Gamma_t d\gamma_t$ coincides with the semi-classical measure in Part (i) and Part (ii) (1)
 of Theorem \ref{theo:schro}, whence the result.
\end{proof}


\section{Uniform $\eps$-oscillation and marginals of semi-classical measures}\label{sec:densitylimit}

We prove here Theorem~\ref{theorem1} in Section \ref{sec:prooftheorem1}, using the  notion of $\eps$-oscillation explained in Section~\ref{subsec:epsosc}, which allows to relate the weak limits of energy density and  marginals of semi-classical measures, that we first study 
 in Section~\ref{subsec_marginal}.

\subsection{Marginals of semi-classical measures}
\label{subsec_marginal}

In this section, we describe the consequence of Theorem~\ref{theo:schro} for the marginals of semi-classical measures for the solutions to the Schr\"odinger equation. 
Let $(\psi^\eps_0)_{\eps>0}$ be a bounded family of $L^\infty(\R,L^2(G))$, the associated semi-classical measure~$\Gamma_t d\gamma_t$ and subsequence $\eps_k$ given by Theorem \ref{theo:schro} whose notation we use.
  We also consider the corresponding measurable map 
$t\mapsto\varrho_t$ in $L^\infty(\R;\mathcal M^+(G))$
given by the marginal  on $G$ of the measure ${\rm Tr}\left( \Gamma_t(x,\lambda)\right) d\gamma_t(x,\lambda)$, that is, formally
$$
\varrho_t := \int_{\widehat G}
{\rm Tr}\left( \Gamma_t\right) d\gamma_t(x,d\lambda),
$$
or more precisely, 
for all
  $\phi\in{\mathcal C}^\infty_c (G)$ and $\theta\in L^1(\R)$, 
$$
\int_{\R\times G} \theta(t) \phi(x) d\varrho_t(x) dt
=\int_{\R}  \theta(t)  
\int_{G\times \widehat G} 
\phi(x)  {\rm Tr} \left(\Gamma_t(x,\lambda) \right) d\gamma_t(x,\lambda) \, dt.
$$
We  define two measurable maps in $L^\infty(\R;\mathcal M^+(G))$  via  
  \begin{align*}
\varrho_t^{\mathfrak z^*}(x)
&:= \int_{{\mathfrak z}^*\setminus\{0\}} {\bf 1}_{\lambda\not=0} {\rm Tr}\, \Gamma_t (x,\lambda) d\gamma_t(x,\lambda),
\\
\varrho_t^{\mathfrak v^*}(x)
&:= \int_{{\mathfrak v}^*} {\bf 1}_{\lambda =0} {\rm Tr}\, \Gamma_t (x,(0,\omega)) d\gamma_t(x,(0,\omega))
= \int_{{\mathfrak v}^*}  \varsigma_t(x,d\omega),
 \end{align*}
 and, because of the decomposition of $\widehat G$ recalled in Section~\ref{subsec_F}, 
 $
\varrho_t=\varrho_t^{\mathfrak z^*}+\varrho_t^{\mathfrak v^*}.$
Besides, Theorem~\ref{theo:schro} directly implies the 
next proposition.

\begin{proposition}\label{prop_theorem1} 
In the setting just above, we have the following properties: 
 \begin{enumerate}
\item 
If $\tau\in (0,1]$,
the maps $t\mapsto \rho_t$, $t\mapsto \rho_t^{\mathfrak z^*}$ 
and $t\mapsto \rho_t^{\mathfrak v^*}$ 
are  weakly continuous from $\R$ to $\mathcal M(G)$.
\item The measures $ \varrho_t^{\mathfrak z^*}$ satisfy the following properties:
 \begin{enumerate}
\item If $\tau\in (0,2)$, $\partial_t \varrho_t^{\mathfrak z^*}=0$ in the sense of distribution on $\R\times G$,
\item If $\tau=2$, 
$$
\varrho_t^{\mathfrak z^*} = \sum_{n\in \N} 
\int_{{\mathfrak z}^*\setminus\{0\}} 
\gamma_{n,t}(x,d\lambda) 
\quad\mbox{where}\quad
\gamma_{n,t} (x,\lambda):= {\bf 1}_{\lambda\not=0} {\rm Tr}\, \Gamma_{n,t} (x,\lambda) d\gamma_t(x,\lambda),
$$ 
and in the sense of distributions on $\R \times G \times (\mathfrak z^*\setminus\{0\})$,
we have 
$$
\left(\partial_t -{2n+d\over 2|\lambda|}  {\mathcal Z}^{(\lambda)}\right)\gamma_{n,t}=0.
$$
\item If $\tau>2$, then 
$\varrho_t^{\mathfrak z^*}=0$,
\end{enumerate}

\item The measures $ \varrho_t^{\mathfrak v^*}$ satisfy the following properties:
\begin{enumerate}
\item If $\tau\in(0,1)$, for all $t\in\R$, 
$ \varrho_t^{\mathfrak v^*}=  \varrho_0^{\mathfrak v^*}$.
\item If $\tau=1$, for all $t\in \R$,
$\displaystyle{
\varrho_t^{\mathfrak v^*} (x) =
\int_{\mathfrak v^*} 
\varsigma_0 \left({\rm Exp} (t\,  \omega\cdot  V ) x,d\omega\right).
}$
\item if $\tau>1$, 
$\varrho_t^{\mathfrak v^*}(x)
=\varsigma_t(x,0)$.
\end{enumerate}
\end{enumerate}
\end{proposition}

This statement is the core of the proof of Theorem~\ref{theorem1} by use of the concept of $\eps$-oscillation that we now discuss.

\subsection{$\eps$-oscillating families}
\label{subsec:epsosc} We stick here to a general framework and consider a family $(u^\eps(t))_{\eps>0}$ bounded in $L^\infty(\R, L^2(G))$. Our aim is  
 to link here the weak limits of the measure $|u^\eps(t,x)|^2 dxdt$ and the semi-classical measures of the family $(u^\eps(t))_{\eps>0}$. 
In analogy with~\cite{gerard_X} and Section~4.4 of~\cite{FF} for H-type groups, we introduce the notion of uniform strict
 $\eps$-oscillations for time-dependent families of $L^\infty(\R, L^2(G))$.

 \begin{definition}\label{def:epsosc}
Let $(u^\eps)_{\eps>0}$ be a bounded family in $L^\infty(\R,L^2(G))$.
We shall say that $(u^\eps)$ is uniformly $\eps$-{\it oscillating} when we have for all $T>0$,
$$
\limsup_{\eps\rightarrow 0} \sup_{t\in[-T,T]} \left\|{\bf 1}_{-\eps^2\Delta_G>R} u^\eps(t)
 \right\|_{L^2(G)}\Tend{R}{+\infty}0.
 $$
If moreover, we have 
$$\limsup_{\eps\rightarrow 0}\sup_{t\in[-T,T]}  \left\| {\bf 1}_{-\eps^2\Delta_G<\delta} 
 u^\eps (t)\right\|_{L^2(G)}\Tend{\delta}{0}0,
 $$
then the family $(u^\eps)$ is said to be uniformly strictly $\eps$-oscillating.
 \end{definition}

The interest of the notion of $\eps$-oscillation relies on the fact that it gives an indication of the size of the oscillations that have to be taken into account. It legitimates the use of semi-classical pseudodifferential operators and semi-classical measures in order  to describe the time-averaged densities~(\ref{def:timeavdens}) of these families. 
 Indeed, we have the following  proposition.

\begin{proposition}
\label{prop:eops}
\begin{enumerate}
\item 
Let $(u^\eps)\in L^\infty(\R, L^2(G))$ be a uniformly $\eps$-oscillating family admitting a time-averaged semi-classical measure $t\mapsto \Gamma_t d\gamma_t$ for the sequence $(\eps_k)_{k\in \mathbb N}$. Then
 for all
  $\phi\in{\mathcal C}^\infty_c (G)$ and $\theta\in L^1(\R)$, 
$$
\lim_{k\rightarrow +\infty}  \int_{\R\times G} \theta(t) \phi(x) |u^{\eps_k}(t,x)|^2 dxdt
=\int_{\R}  \theta(t)  
\int_{G\times \widehat G} 
\phi(x)  {\rm Tr} \left(\Gamma_t(x,\lambda) \right) d\gamma_t(x,\lambda) \, dt,
$$
\item	If moreover $(u^\eps)$ is  uniformly strictly $\eps$-oscillating family, then the semi-classical measure does not charge the trivial representation $1_{\widehat G}$ in the sense that
$$
\gamma_t( G\times \{1_{\widehat G}\})=0
\quad \mbox{for almost every} \ t\in \R.
$$
\end{enumerate}
\end{proposition}

\begin{proof}  
Let $\phi\in{\mathcal C}^\infty_c (G)$ and let $\theta\in L^1(\R)$. 
 We can write for any $R>0$
  $$
 \int_{\R\times G} \theta(t) \phi(x) |u^{\eps}(t,x)|^2 dt dx = I_{1,\eps,R} +  I_{2,\eps,R},
 $$
 where
 $$
I_{j,\eps,R}:=\int_{\R} \theta(t)  
\left({\rm Op}_\eps (\sigma_{j,R}) u^\eps(t),u^\eps(t)\right)_{L^2(G)} dt,
\quad j=1,2,
$$
with
$$
\sigma_{1,R} (x,\lambda):=  \phi(x) \, \chi\left({\eps^2\over R}H(\lambda)\right)
\qquad\mbox{and}\qquad
 \sigma_{2,R} (x,\lambda):=  \phi(x) \, (1-\chi)\left({\eps^2\over R}H(\lambda)\right),
 $$
having fixed a function $\chi\in{\mathcal C}^\infty(\R)$  such that $0\leq \chi\leq 1$, $\chi=0$ on $]-\infty,1]$ and $\chi=1$ on $[2,+\infty[$. 
Note that the symbol $\sigma_{2,R}$ is in ${\mathcal A}_0$
while, using the notation of \cite{FF}, the symbol $\sigma_{1,R}$ is in  $S^0$. 
 For $I_{1,\eps,R}$, let us first assume that    $\theta$ is compactly supported in $[-T,T]$ for some $T>0$. We have
$$
|I_{1,\eps,R}| 
\leq 
\|\theta\|_{L^1(\R)} 
\|\phi\|_{L^\infty(G)}
\|u^\eps\|_{L^\infty(\R,L^2(G))}
\sup_{t\in [-T,T]}\left\|\chi\left(-\eps^2 R^{-1}\Delta_G\right) u^\eps\right\|_{L^2(G)},
$$
so 
$$
\lim_{R\to+\infty}\limsup_{k\to\infty}I_{1,\eps_k,R} =0.
$$
since  $(u^\eps)$ is $\eps$-oscillating and $0\leq \chi\leq {\bf 1}_{x>1}$.
By density of $C_c^\infty(\R)$ in $L^1(\R)$,  this is also true for any $\theta\in L^1(\R)$.

\noindent For $I_{2,\eps,R}$, 
by Theorem \ref{theo:mesures},
we have
$$
\lim_{k\to \infty}
I_{2,\eps_k,R}  
=
\int_{\R}  \theta(t)  
\int_{G\times \widehat G} 
 {\rm Tr} \left(\sigma_{2,R}(x,\lambda)\Gamma_t(x,\lambda) \right) d\gamma_t(x,\lambda) \, dt .
$$
Since we have
$$
|{\rm Tr} \left(\sigma_{2,R}(x,\lambda)\Gamma_t(x,\lambda) \right)| \leq \|\phi\|_{L^\infty(G)}{\rm Tr} \left(\Gamma_t(x,\lambda) \right), 
$$
and
$$
\lim_{R\to +\infty}{\rm Tr} \left(\sigma_{2,R}(x,\lambda)\Gamma_t(x,\lambda) \right)
=
\phi(x){\rm Tr} \left(\Gamma_t(x,\lambda) \right),
$$ 
the Lebesgue dominated convergence theorem implies
$$
\lim_{R\to+\infty}\lim_{k\to\infty}I_{2,\eps_k,R} =
\int_{\R}  \theta(t)  
\int_{G\times \widehat G} 
\phi(x)  {\rm Tr} \left(\Gamma_t(x,\lambda) \right) d\gamma_t(x,\lambda) \, dt .
$$
This yields Part (1).

\medskip

For Part (2), we see that if 
 moreover $(u^\eps)$ is  uniformly strictly $\eps$-oscillating family,
 then  for any $\theta\in C_c(\R)$ and $\phi\in C_c^\infty(G)$, 
 the expression
 $$
 \int_{\R} \theta(t)  (\phi (x) (1-\chi) (-\frac{\eps^2}\delta \Delta_G) u^\eps (t), u^\eps (t))_{L^2(G)} dt
 $$
 is bounded by 
$$
\|\theta\|_{L^1(\R)}
\sup_{t\in {\rm supp} \theta} \| (1-\chi) (-\frac{\eps^2}\delta \Delta_G) u^\eps (t)\|_{L^2(G)} 
\sup_{t\in \R} \|u^\eps (t)\|_{L^2(G)}  \|\phi\|_{L^\infty(G)} ,
$$
 which tends to 0 when $\eps=\eps_k$ with $k\to+\infty$ and then $\delta\to 0$.
 However, by Theorem \ref{theo:mesures}, the limit of the same expression as $\eps=\eps_k$ with $k\to+\infty$ is 
 $$
 \int_{\R} \theta \int_{G\times \widehat G} 
 {\rm Tr}\left( \phi(x) (1-\chi)(-\frac 1 \delta \mathcal F \Delta_G (\lambda))\Gamma_t(x,\lambda)  \right)
d\gamma_t(x,\lambda),
$$
which, by Lebesgue's dominated convergence theorem, 
converges as $\delta\to 0$ to 
$$
 \int_{\R} \theta (t)
 \int_{G\times \widehat G} \phi(x) 1_{\pi=1_{\widehat G}}
d\gamma_t (x,\lambda) ,
$$
since $(1-\chi)(\frac 1 \delta \mathcal F \Delta_G)(\pi) $ tends to 0 in SOT for any non-trivial representation $\pi \in \widehat G$ while at $\pi=1_{\widehat G}$ it is equal to 1.
Consequently this last expression is zero, and this concludes the proof of Proposition \ref{prop:eops}.
\end{proof}

The fact that a family is uniformly $\eps$-oscillating can  be derived from Sobolev bounds. 

\begin{proposition}
\label{prop:sobcri}
\begin{itemize}
	\item If there exists $s>0$ and $C>0$ such that 
$$\forall \eps>0,\;\; \sup_{t\in[-T,T]} \| (-\eps^2\Delta_G)^{s\over 2} u^\eps(t) \|_{L^2(G)}
\leq C,$$
then $(u^\eps)_{\eps}$ is $\eps$-oscillating.
\item If there exists $s>0$ and $C>0$ such that 
$$\forall \eps>0,\;\; \sup_{t\in[-T,T]} \| (-\eps^2\Delta_G)^{s\over 2} u^\eps(t) \|_{L^2(G)}
+\sup_{t\in[-T,T]} \| (-\eps^2\Delta_G)^{-{s\over 2}} u^\eps(t) \|_{L^2(G)}
\leq C,$$
then $(u^\eps)_{\eps}$ is strictly $\eps$-oscillating.
\end{itemize}
\end{proposition}

\begin{proof}
We use Plancherel formula~\eqref{Plancherelformula} and the facts that for $s>0$, 
$$ \displaylines{
 \chi \left(-{\eps^2\over R} \Delta_G\right) \leq {(-\eps^2\Delta_G)^{s\over 2}\over R^s}  \chi \left(-{\eps^2\over R} \Delta_G\right)\leq {(-\eps^2\Delta_G)^{s\over 2}\over R^s}\cr
\mbox{and}\;\; 
 (1- \chi) \left(-{\eps^2\over \delta} \Delta_G\right) \leq \delta^{s} (-\eps^2\Delta_G)^{-{s\over 2}} (1-\chi) \left(-{\eps^2\over \delta} \Delta_G\right)\leq \delta^{s} (-\eps^2\Delta_G)^{-{s\over 2}} .\hfill \cr}$$
\end{proof}

For families of solutions to  the Schr\"odinger equation~(\ref{eq:schrosc}), the uniform $\eps$-oscillating property is inherited from the initial data. Indeed,  
 using that operators of the form $ \chi \left(-{\eps^2\over R} \Delta_G\right) $ 
 commute with the sublaplacian, 
the following result follows by energy estimates as a consequence of Proposition~\ref{prop:sobcri}:

\begin{proposition}
\label{prop:tbytepsosc}
Let $(\psi^\eps_0)_{\eps>0}$ be a bounded family of $L^\infty(\R,L^2(G))$.
If it satisfies
$$
\exists s,C>0,\qquad \forall \eps>0\qquad  
 \| (-\eps^2\Delta_G)^{s\over 2} \psi^\eps_0\|_{L^2(G)} 
\leq C,
$$
then the family of solutions $\psi^\eps(t)$ to the equation~(\ref{eq:schrosc}) for the initial data $(\psi^\eps_0)$ is uniformly  $\eps$-oscillating. Moreover, if $(\psi^\eps_0)$ satisfies~(\ref{Hscriterium}), then $(\psi^\eps(t))$ is uniformly strictly $\eps$-oscillating. 
  \end{proposition}

From Propositions~\ref{prop:eops} (2) and~\ref{prop:tbytepsosc}, it follows that:

\begin{corollary}
\label{cor_tau>1}
We continue with the setting and the notation in  Proposition \ref{prop_theorem1}.
Assume in addition that  $(\psi^\eps_0)$ satisfies~(\ref{Hscriterium}).
If  $\tau>1$ then $\varrho_t^{\mathfrak v^*}=0$.
\end{corollary}

\subsection{Proof of Theorem \ref{theorem1}}
\label{sec:prooftheorem1}

Let $(\psi^\eps_0)_{\eps>0}$ be a bounded family of $L^\infty(\R,L^2(G))$ satisfying \eqref{Hscriterium}.
We set $\psi^\eps(t)= {\rm e}^{i {t\over 2\eps^{\aleph} } \Delta_G}\psi^\eps_0$. Then $\psi^\eps(t)$ satisfies the semi-classical Schr\"odinger equation~(\ref{eq:schrosc}) with $\tau=\aleph+2$,
and $(\psi^\eps_t)$ is uniformly strictly $\eps$-oscillating by Proposition \ref{prop:tbytepsosc}.

We consider a weak limit of $| {\rm e}^{i {t\over 2\eps^{\aleph} }\Delta_G}\psi^\eps_0(x)|^2 dx\,dt$ for a converging subsequence $(\eps_j)$.
 Up to another extraction of a subsequence, it admits a semi-classical measure $\Gamma_t d\gamma_t$ as in Theorem \ref{theo:schro}.
 By Proposition \ref{prop:eops}, 
 the marginals $\rho_t$ defined in Proposition \ref{prop_theorem1} coincide with the weak limit of $| {\rm e}^{i {t\over2 \eps^{\aleph} } \Delta_G}\psi^\eps_0(x)|^2 dx\,dt$. 
The result now readily follows from Proposition \ref{prop_theorem1} and Corollary \ref{cor_tau>1}.

\begin{remark}
 The case $\aleph \in (-2,-1]$ in Theorem \ref{theorem1} holds under 	the weaker hypothesis 
$$
\exists s,C>0,\;\;   \| (-\eps^2\Delta_G)^{s\over 2} \psi^\eps_0\|_{L^2(G)}\leq C.
$$
\end{remark}

\appendix
\section{Dispersion in the Euclidean case}

We describe here the analogue of Theorem~\ref{theorem1} in the Euclidean setting for the Laplace operator $\Delta= \sum_{1\leq j\leq d} \partial_{x_j}^2$.  The assumption~(\ref{Hscriterium}) then writes in the same manner replacing the sub-Laplacian by  the Laplace operator $\Delta$.  We point out that the result below is only an elementary version of results that hold in more general setting and for more general Hamiltonian, including integrable systems (see~\cite{AFM,CFM}). We use the semi-classical measures as introduced in the 90's in~\cite{gerard_X,gerardleichtnam,GMMP,LionsPaul}.  

\begin{lemma}
Let $(\psi^\eps_0)$ be a bounded family in $L^2(\R^d)$ satisfying~(\ref{Hscriterium}). Then any limit point of the measure $\left|{\rm e}^{-i{t\over 2\eps^\aleph} \Delta}\psi^\eps_0\right| ^2 dxdt$  is of the form $\varrho_t(x) dt$ where $\varrho_t$ is a measure on $\R^d$.  Besides
\begin{enumerate}
\item If $\aleph\in(-2,-1)$, then $\partial_t \varrho_t =0$.
\item If $\aleph=-1$ then $\varrho_t(x)=\int_{\R^d} \mu_0(x-t\xi,d\xi)$.
\item If $\aleph >-1$ then $\varrho_t=0$.
\end{enumerate} 
\end{lemma}

\begin{proof}
A simple proof of this fact can be given by use of semi-classical measures. Denoting by~${\rm Op}_\eps(a)$ the semi-classical pseudodifferential operator of symbol $a$, a semi-classical measure~$\mu_t$  of the family $\psi^\eps(t,x):={\rm e}^{-i{t\over 2\eps^\aleph} \Delta}\psi^\eps_0$ is such that for a subsequence $\eps_k$,  for all $\theta\in{\mathcal C}_c^\infty(\R)$ and $a\in{\mathcal C}_c^\infty(\R^{2d})$,
\begin{equation}\label{eq:annex1}
\left({\rm Op}_{\eps_k}(a) \psi^{\eps_k}(t)\;,\;\psi^{\eps_k}(t) \right) \Tend{k}{+\infty} \int_{\R\times \R^{2d}} \theta(t) a(x,\xi) \mu_t(dx,d\xi) dt.
\end{equation}
The fact that $(\psi^\eps_0)$  satisfies~(\ref{Hscriterium}) implies that it  is a strictly $\eps$-oscillating family and that it is also the case for  the family $\psi^\eps(t)$. One then has  $\mu_t(\{\xi=0\}=0$ and 
 for the subsequence $\eps_k$ of~\eqref{eq:annex1}, 
$\theta\in{\mathcal C}_c^\infty(\R)$ and $\phi\in{\mathcal C}_c^\infty(\R^d)$,
$$\int_{\R\times\R^d} \theta(t) \phi(x) |\psi^{\eps_k}t,x)|^2 dxdt \Tend{k}{+\infty} \int_{\R \times \R^{2d}} \theta(t) \phi(x) \mu_t (dx,d\xi)dt.$$
The knowledge of the semi-classical measures determines all  the limit points of the energy density.

\medskip 

Let us now take
 $a\in{\mathcal C}_c^\infty(\R^{2d})$, we observe that 
$$
{d\over dt} \left({\rm Op}_\eps(a) \psi^\eps(t)\;,\;\psi^\eps(t) \right) = {1\over i\eps^\kappa} \left(\left[ {\rm Op}_\eps(a)\;,\; -{\eps^2\over 2} \Delta\right] \psi^\eps(t) \;,\;\psi^\eps(t)\right).
$$
Since 
$$
\left[ {\rm Op}_\eps(a)\;,\; -{\eps^2\over 2} \Delta\right]=i\eps \, {\rm Op}_\eps(\xi\cdot \nabla a) + \eps^2 {\rm Op}_\eps(\Delta a).$$
We obtain immediately the following description:
\begin{enumerate} 
\item For $\kappa\in(0,1)$, 
$$
\left({\rm Op}_\eps(a) \psi^\eps(t)\;,\;\psi^\eps(t) \right)  = \left({\rm Op}_\eps(a) \psi^\eps_0\;,\;\psi^\eps_0 \right) + O(\eps^{1-\kappa}),
$$
whence $\mu_t(x,\xi)=\mu_0(x,\xi)$ for all times $t\in\R$.
\item For $\kappa=1$,  
the map $t\mapsto \mu_t$ is weakly continuous form $\R$ to $\mathcal M^+(\R^{2n})$ and can be realised by the same subsequence $\eps_k$ for all $t\in[0,T]$, $T>0$ with
$$
\partial_t \mu_t(x,\xi) = \xi\cdot \nabla_x \mu_t(x,\xi)
\qquad \mbox{in the sense of distributions}.
$$
\item For $\kappa>1$, we observe that
\begin{eqnarray*}
\left.{d\over ds} \left({\rm Op}_\eps(a(x+s\xi,\xi)) \psi^\eps(t)\;,\;\psi^\eps(t) \right) \right|_{s=0}& =& \left({\rm Op}_\eps(\xi\cdot a)\psi^\eps(t)\;,\;\psi^\eps(t) \right)\\
& = &\eps^{\kappa-1} {d\over dt}  \left({\rm Op}_\eps(a) \psi^\eps(t)\;,\;\psi^\eps(t) \right) + O(\eps).
\end{eqnarray*}
\end{enumerate}
In the last case, we deduce that 
 for $\theta\in{\mathcal C}_c^\infty(\R)$,
$$
\displaylines{\qquad  \int\theta(t) \left.{d\over ds} \left({\rm Op}_\eps(a(x+s\xi,\xi)) \psi^\eps(t)\;,\;\psi^\eps(t) \right) \right|_{s=0}dt\hfill\cr\hfill 
 = -i\eps^{\kappa-1} \int\theta'(t) 
  \left({\rm Op}_\eps(a) \psi^\eps(t)\;,\;\psi^\eps(t) \right)dt = O(\eps^{\kappa-1}).\qquad\cr}
  $$
  Therefore, the measure $\mu_t$ is invariant under the flow
  $(x,\xi) \mapsto (x+s\xi,\xi)$
  and,  since $\mu_t$ is of finite mass, $\mu_t $ is supported on $\{\xi=0\}$, whence $\mu_t=0$ by the strict $\eps$-oscillating assumption
\end{proof}

\section{Proof of Lemma~\ref{lem_sigma_comH} and Hermite functions}
\label{sec_Hermite+pflem_sigma_comH}

The proof of Lemma \ref{lem_sigma_comH} uses the bracket structure of $\mathfrak g$ 
via the two following lemmata:

\begin{lemma}
 \label{lem_comDeltaPQ}
For any	$\lambda\in \mathfrak z^* \setminus\{0\}$ and $j=1,\ldots, d$, we have:
$$ 
[\Delta_G,P_j]=-2|\lambda|^{-1}{\mathcal Z}^{(\lambda)} Q_j
 \quad\mbox{and}\quad	
 [\Delta_G,Q_j]=2|\lambda|^{-1} {\mathcal Z}^{(\lambda)} P_j.
$$
 \end{lemma}
  
\begin{proof}[Proof of Lemma \ref{lem_comDeltaPQ}]
By \eqref{eq_comPQ}, $P_{j_0}$ commutes with $Q_j$ if $j\not =j_0$ and with any $P_j$,
so \eqref{eq_DeltaGPQ}	yields
$$
[\Delta_G,P_{j_0}] = 
 \sum_{j = 1}^{d} \left([P_j^2,P_{j_0}] + [Q_j^2,P_{j_0}]\right)
 = 
 [Q_{j_0}^2,P_{j_0}]
 =
 Q_{j_0}[Q_{j_0},P_{j_0}]
 +[Q_{j_0},P_{j_0}]Q_{j_0}.
$$ 
Using now \eqref{eq_Zlambda}, we obtain the first equality of the statement. The second equality is proved in a similar way.	
\end{proof}

\begin{lemma}
\label{lem_compT}
	Considering for $\lambda\in \mathfrak z^*\setminus\{0\}$,
$$
T:= 
\left(\sum_{j_1=1}^{2d}
V_{j_1} \pi^\lambda (V_{j_1})\right)
\left(\sum_{j_2=1}^d
\left(P_{j_2} \pi^\lambda (Q_{j_2}) - Q_{j_2} \pi^\lambda(P_{j_2})\right)\right),
$$
we have:
\begin{equation*}\label{eq:PinTPin}
\Pi_n^{(\lambda)}  T \Pi_n^{(\lambda)}  = 
\frac{|\lambda|}2 \left( |\lambda|^{-1} {\mathcal Z}^{(\lambda)} (2n+d)
 +i \Delta_G\right) \Pi_n^{(\lambda)} .
 \end{equation*}
 \end{lemma}
 
The proof of \ref{lem_compT} will use the properties of a special family of $H(\lambda)$-eigenfunctions we now recall.
 The family of Hermite functions $(h_n)_{n\in\N}$ given by 
 $$
 h_n(\xi) = \frac{(-1)^n}{\sqrt{2^n n! \sqrt \pi}}
 e^{\frac {\xi^2}2}
 \frac{d }{d\xi} ( e^{\xi^2}), \quad n\in \N, 
 $$
 is an orthonormal basis of~$L^2(\R) $ which satisfies 
$$
-h''_n(\xi)+\xi^2 h_n(\xi)= (2n+1) h_n(\xi) \, .
$$  
Hence, for each multi-index~$\alpha \in {\mathbb N}^d$, 
the function $h_{\alpha}$ 
defined  by
$$
 h_{\alpha} (\xi)  :=\prod_{j=1}^d h_{\alpha_j}(\xi_j),  \;\;\xi = (\xi_1,\dots, \xi_d) \in \R^d,
$$
is an eigenfunction of the operator $H(\lambda)$ (see Section \ref{freq}):
$$
H(\lambda) h_{\alpha} =    |\lambda| (2|\alpha|+d) 
\  h_{\alpha}.
$$
The eigenvalues $|\lambda| (2|\alpha|+d)$, $\alpha\in \N^d$, describe the entire spectrum of $H(\lambda)$
since the functions $h_\alpha$, $\alpha\in \N^d$ form an orthonormal basis of $L^2(\R^d)$.

For each $\lambda\in \mathfrak z^*\setminus\{0\}$, 
the symplectic structure on $\mathfrak v$ given by $B(\lambda)$ naturally suggests to consider a complex structure by setting: 
$$
R_j := \frac12 (P_j - iQ_j), 
\quad\mbox{and}\quad 
\bar R_j := \frac12 (P_j + iQ_j).
$$
By \eqref{eq_piPQZ}, the operators 
$$
\pi^\lambda (R_j)=\frac{\sqrt{|\lambda|}}2 (\partial_{\xi_j} +\xi_j)
\;\;\mbox{and}\;\;
\pi^\lambda (\bar R_j)=\frac{\sqrt{|\lambda|}}2 (\partial_{\xi_j} -\xi_j) 
$$
are the creation-anihilation operators  associated with the harmonic oscillator $H(\lambda)$.
The well known recursive relations of the Hermite functions 
$$
 h'_n(\xi) = \sqrt{\frac n 2} h_{n-1}(\xi) - \sqrt{\frac{n+1}2} h_{n+1}(\xi),
\qquad
\xi h_n (\xi)= \sqrt{\frac n 2} h_{n-1}(\xi) + \sqrt{\frac{n+1}2} h_{n+1}(\xi),
$$
	gives 
for each $\lambda\in {\mathfrak z}^*\setminus\{0\} $, 
$n\in \N$ and $j=1,\ldots,d$,
 $$
\pi^\lambda (R_j) h_\alpha
= \frac{\sqrt{|\lambda|}}2 
\sqrt{2\alpha_j} h_{\alpha-{\bf 1}_j}
\qquad
\pi^\lambda(\bar R_j) h_\alpha
= -\frac{\sqrt{|\lambda|}}2 
\sqrt{2(\alpha_j+1)} h_{\alpha+{\bf 1}_j}.
$$
Consequently, we have
 $
 \pi^\lambda(\bar R_j)({\mathcal V}_n)={\mathcal V}_{n+1}$
and
$ \pi^\lambda(R_j)({\mathcal V}_n)={\mathcal V}_{n-1}
 $
with the convention that ${\mathcal V}_{-1}=\{0\}$. Moreover 
    $$
  \pi^\lambda(R_j) \pi^\lambda(\bar R_j) h_\alpha
  = 
  -\frac{|\lambda|}2 (\alpha_j+1) h_\alpha
  \qquad
 \pi^\lambda(\bar R_j) \pi^\lambda(R_j) h_\alpha
 =
  -\frac{|\lambda|}2 \alpha_j h_\alpha
 $$

\begin{proof}[Proof of Lemma \ref{lem_compT}]
We observe that
\begin{equation}\label{eq:Annex2}
\sum_{j=1}^{2d} V_j \pi^\lambda(V_j)= \sum_{j=1}^d(P_j\pi^\lambda(P_j)+Q_j \pi^\lambda(Q_j) ) = 
\sum_{j=1}^d (R_j\pi^\lambda(\bar R_j)+\bar R_j \pi^\lambda(R_j)),
\end{equation}
whence
$$
T=
-4i 
\sum_{ j_1,j_2 }
\left( R_{j_1}  \pi^\lambda(\bar R_{j_1}) + \bar R_{j_1} \cdot \pi^\lambda(R_{j_1})\right) \left(R_{j_2} \pi^\lambda (\bar R_{j_2}) - \bar R_{j_2} \pi^\lambda(R_{j_2})\right).
$$
Hence the properties of the $R_j$'s given above yield 
\begin{align*}
\Pi_n T \Pi_n
&=
-4i\sum_{j=1}^d
\Pi_n \left(-R_j\bar R_j \pi^\lambda(\bar R_j)\pi^\lambda (R_j) 
 + \bar R_j R_j \pi^\lambda( R_j)\pi^\lambda (\bar R_j)
 \right)\Pi_n
\\
\nonumber
&=
-4i	\sum_{j=1}^d  \sum_{|\alpha|=n}
\Pi_n \left(\frac{|\lambda|}2 \alpha_j  R_j\bar R_j    
  -\frac{|\lambda|}2 (\alpha_j+1) \bar R_j R_j 
 \right) \Pi_\alpha
\end{align*}
 since $\Pi_n =\sum_{|\alpha|=n} \Pi_\alpha$ where $\Pi_\alpha= |h_\alpha \rangle \, \langle h_\alpha|$.
We compute 
$$
\sum_{j=1}^d \left(\alpha_j  R_j\bar R_j    
 - (\alpha_j+1) \bar R_j R_j 
 \right) 
= 
\frac i4 |\lambda|^{-1} {\mathcal Z}^{(\lambda)} (2|\alpha|+d)
 -  \frac 14 \Delta_G ,
$$
and the result follows.
\end{proof}

We can now show Lemma \ref{lem_sigma_comH}.

\begin{proof}[Proof of Lemma \ref{lem_sigma_comH}]
By Lemma \ref{lem_g(lambda)} Part (1), $\sigma\in {\mathcal A}_0$.
 To prove Part (1), we first observe that,
 since the endomorphism $B(\lambda)$ defined via \eqref{skw}
is represented by $J$ in the $(P_1,\ldots,P_d,Q_1,\ldots,Q_d)$-basis,
 we have in vector notation 
\begin{align*}
\sum_{j=1}^d \left(P_j \pi^\lambda (Q_j) - Q_j \pi^\lambda(P_j)\right)
&=
\left(\begin{array}{c} P\\Q\end{array}\right)^t 
J \left(\begin{array}{c} \pi(P)\\\pi(Q)\end{array}\right)
=
|\lambda|^{-1} V^t \ B(\lambda) \pi(V)
\\&=
|\lambda|^{-1} \sum_{j,k} B(\lambda)_{j,k} V_j \pi(V_k).
\end{align*} 
Hence we can write 
 $$
\sigma_1(x,\lambda)
=\frac{-1} {2i|\lambda|^{2}}\sum_{j,k=1}^{2d}
 B(\lambda)_{j,k} \pi^\lambda(V_k) 
V_j \sigma(x,\pi)
=\sum_{j,k=1}^{2d}
\pi^\lambda(V_k) V_j g_{j,k}(\lambda)\sigma_0(x,\pi),  
$$
As $g$ is smooth and supported away from 0
and $B(\lambda)$ depends linearly on $\lambda\in {\mathfrak z}^*\setminus\{0\}$,
each function $g_{j,k}:=B_{j,k} g$ is smooth on ${\mathfrak z}^*\setminus\{0\}$; 
it is also Schwartz as $g$ is Schwartz.
 By Lemma \ref{lem_g(lambda)} Part (1), each symbol $g_{j,k} (\lambda) \sigma_0$ is  in ${\mathcal A}_0$  
 so $\sigma_1\in {\mathcal A}_0$.
 This shows Part (1).
 
 \medskip 
 
 Part (2) follows from the  observation that, 
 as $\sigma$ commutes with $H(\lambda)$, 
 we have:
 $$
\left[\sigma_1(x,\lambda), H(\lambda)\right]
=\frac{-1}{2i|\lambda|} 
\left(\sum_{j=1}^d P_j \left[\pi^\lambda (Q_j),H(\lambda)\right] - Q_j \left[\pi^\lambda(P_j),H(\lambda)\right]\right)\sigma(x,\lambda).$$
We then use Lemma \ref{lem_comDeltaPQ} and write 
$$
P_j \left[\pi^\lambda (Q_j),H(\lambda)\right] - Q_j \left[\pi^\lambda(P_j),H(\lambda)\right]
= -2|\lambda|^{-1} \pi^\lambda({\mathcal Z}^{(\lambda)}) ( P_j\pi^\lambda (P_j)+Q_j\pi^\lambda (Q_j)),
$$ 
which allows us to conclude in view of \eqref{eq_piPQZ} and \eqref{eq:Annex2}.

 \medskip 
 
  Part (3) follows from 
the commutativity of $\sigma$ with $H$ 
and Lemma \ref{lem_compT}.	
\end{proof}


\end{document}